\documentclass{article}
\usepackage{PRIMEarxiv}

\usepackage{todonotes}
\usepackage{multirow}
\usepackage{algorithm}
\usepackage{algpseudocode}
\usepackage{amsthm}
\newtheorem{theorem}{Theorem}

\newtheorem{lemma}{Lemma}
\newtheorem{remark}{Remark}
\newtheorem{definition}{Definition}
\newtheorem{conjecture}{Conjecture}
\newtheorem{proposition}{Proposition}

\makeatletter
\renewenvironment{proof}[1][\proofname]{\par
  \vspace{-\topsep}
  \pushQED{\qed}%
  \normalfont
  \topsep0pt \partopsep0pt 
  \trivlist
  \item[\hskip\labelsep
        \itshape
    #1\@addpunct{.}]\ignorespaces
}{%
  \popQED\endtrivlist\@endpefalse
  \addvspace{6pt} 
}
\makeatother

\usepackage[utf8]{inputenc} 
\usepackage[T1]{fontenc}    
\usepackage{hyperref}       
\usepackage{url}            
\usepackage{booktabs}       
\usepackage{amsfonts}       
\usepackage{nicefrac}       
\usepackage{microtype}      
\usepackage{xcolor}         
\usepackage{caption}
\usepackage{subcaption}
\usepackage[numbers,sort]{natbib}

\usepackage{amsmath, bm, amssymb}
\usepackage{empheq}
\PassOptionsToPackage{pdftex}{graphicx}

\title{Beyond adjacency: Graph encoding with reachability and shortest paths
}

\author{
    Shiqiang Zhang,~~Ruth Misener$^\star$\\
    Department of Computing, Imperial College London
}

\begin{document}
\maketitle

\def\thefootnote{$\star$}\footnotetext{Corresponding author: \texttt{r.misener@imperial.ac.uk}}

\begin{abstract}
    Graph-structured data is central to many scientific and industrial domains, where the goal is often to optimize objectives defined over graph structures. Given the combinatorial complexity of graph spaces, such optimization problems are typically addressed using heuristic methods, and it remains unclear how to systematically incorporate structural constraints to effectively reduce the search space. This paper introduces explicit optimization formulations for graph search space that encode properties such as reachability and shortest paths. We provide theoretical guarantees demonstrating the correctness and completeness of our graph encoding. To address the symmetry issues arising from graph isomorphism, we propose lexicographic constraints over neighborhoods to eliminate symmetries and theoretically prove that adding those constraints will not reduce the original graph space. Our graph encoding, along with the corresponding symmetry-breaking constraints, forms the basis for downstream optimization tasks over graph spaces.
\end{abstract}

\section{Introduction}\label{sec:intro}
Graphs ubiquitously appear across science and industry, leading to various decision-making tasks over graph spaces, e.g., computer-aided molecular design (CAMD) \citep{ng2014challenges}, neural architecture search (NAS) \citep{white2023NASthousand}, and topological optimization \citep{maze2023diffusion}. Due to the combinatorial nature of graphs, optimization over graph spaces is usually conducted via sample- or evolution-based algorithms, following a sample-then-evaluate procedure \citep{wan2021adversarial,ru2021interpretable,white2021bananas}. Those works spend substantial effort on developing specific heuristics to efficiently explore the graph space and find promising solutions. With the advances of machine learning (ML), numerous graph ML models, e.g., graph neural networks (GNNs) \citep{wu2020comprehensive,zhou2020graph}, are proposed to learn graph representations for prediction or classification tasks \citep{chami2022machine}. For generation purposes, generative models, e.g., variational autoencoders (VAEs) \citep{kipf2016variational}, generative adversarial networks (GANs) \citep{wang2018graphgan}, normalizing flows \citep{kobyzev2020normalizing}, are gradually extended to graph domain \citep{faez2021deep,liu2023generative}. To avoid direct optimization over discrete spaces, those works often approximate graph optimization over a continuous latent space and then decode the continuous latent variables to graphs. Although they achieves promising performance in various graph-based applications, we observe the following limitations: (i) feasibility of solutions is usually learned from data and checked later by evaluating problem-specific requirements, (ii) optimality is not guaranteed since solutions are often obtained via applying heuristics over original search space or solving an approximated problem over a latent space, (iii) those existing methods are are difficult to generalize to other tasks since they are highly problem-specific and mostly data-driven. 

The aforementioned issues motivate us to take a step back and rethink how to optimize over graph spaces using discrete optimization. For instance, we may formulates trained GNNs using mixed integer programming (MIP) to globally optimize problems with GNNs as objectives or constraints, and thereby verify GNN robustness \citep{hojny2024verifying} and find optimal molecular graphs \citep{zhang2024augmenting}. Symmetry issues caused by graph isomorphism may be resolved by symmetry-breaking constraints, which removes most symmetries and accelerates the solving process in molecular design problems \citep{zhang2023optimizing}. Molecular space may be encoded using integer programming (IP) and molecular validity may be restricted by directly adding mathematical constraints into the formulation \citep{zhang2025limeade}. Recently, graph Bayesian optimization (BO), relying on advances in graph kernels \citep{vishwanathan2010graph, borgwardt2020graph, cao2025metric} and BO \citep{frazier2018tutorial, schulz2018tutorial}, emerges as an attractive approach to optimize black-box graph functions \citep{ru2021interpretable, wan2021adversarial}. We develop MIP formulations for shortest-path graph kernels \citep{borgwardt2005shortest} to globally optimize acquisition functions in graph BO and numerically show the efficiency and potential of our methods in both CAMD \citep{xie2025molecular,xie2025bogrape} and NAS \citep{xie2025nasgoat}. The core of our formulations is encoding of graph spaces, e.g., connected, undirected graphs for molecules, weakly connected, acyclic, directed graphs (digraphs) for neural networks.

This paper aims to distill the fundamental contributions of our recent works \citep{zhang2023optimizing,xie2025bogrape,xie2025molecular,xie2025nasgoat} and extend them from specific applications to general settings. First, we use IP to formulate the graph space consisting of all labeled graphs, no matter their directionality or connectivity. Besides the classic adjacency variables controlling the existence of edges, we additionally encode reachability and shortest paths information, which allows optimization involving these properties, e.g., graph BO equipped with shortest path graph kernels. Our general encoding is easily limited to special subspaces, e.g., undirected graphs, strongly connected digraphs, and acyclic digraphs (DAGs), by simply adding certain linear constraints without introducing extra variables. Assuming strong connectivity, our encoding could be significantly simplified at the cost of losing control of the shortest paths between arbitrary pair of nodes. Both the general and the simplified encoding could be applied to the underlying graph space, i.e., removing directions of all edges, to guarantee weak connectivity of the original graph space, which is more challenging comparing to strong connectivity or acyclicity. When graphs are unlabeled, the symmetry issue appears, i.e., there are $n!$ ways to index a graph with $n$ nodes, while each indexing is a different solution to graph optimization. To handle this issue, we propose symmetry-breaking techniques to remove symmetric solutions to general graphs and design better constraints for DAGs. We theoretically guarantee that adding those constraints will not reduce the real search space, and numerically show that our constraints can remove most symmetries. The code is available at \href{https://github.com/cog-imperial/graph_encoding}{https://github.com/cog-imperial/graph\_encoding}.

\textbf{Paper structure:} Section \ref{sec:general_encoding} presents the general graph encoding from previous work \citep{xie2025bogrape,xie2025molecular,xie2025global}. Section \ref{sec:simplified_encoding} proposes a simplified encoding after assuming strong connectivity, reducing the magnitude of variables and constraints needed from $O(n^3)$ to $O(n^2)$, where $n$ is the graph size. Section \ref{sec:symmetry_breaking} improves our symmetry-breaking work \citep{zhang2023optimizing} by removing the need of one indexed node and introduces more powerful constraints for DAGs. Section \ref{sec:results} evaluates the performance of our symmetry-breaking constraints in several settings. Section \ref{sec:conclusion} summarizes the paper and discusses future works. Appendices \ref{app:example_undirected} and \ref{app:example_DAG} provide two examples illustrating how to use our graph indexing Algorithms \ref{alg:indexing_undirected} and \ref{alg:indexing_DAG} to index a general graph and a DAG, respectively.

\begin{table}[]
    \centering
    \caption{List of popular graph spaces that could be formulated using our graph encoding. Assume that all graphs are labeled. $(\cdot)^U$ ($(\cdot)^T$, respectively) means applying Eq.~($\cdot$) to the underlying (transpose, respectively) graph space.}
    \label{tab:summary_encoding}
    \begin{tabular}{ccccc}
        \toprule
        Space & Description & Encoding & Simplified encoding & OEIS \\
        \midrule
        $\mathcal G_{n_0,n}^C$ & connected undirected graphs & \eqref{eq:general_encoding}+\eqref{eq:undirect} & \eqref{eq:simplified_encoding}+\eqref{eq:undirect} & \href{https://oeis.org/A001187}{\texttt{A001187}}\\
        $\mathcal G_{n_0,n}^S$ & strongly connected digraphs & \eqref{eq:general_encoding}+\eqref{eq:strong_connectivity} & \eqref{eq:simplified_encoding}+\eqref{eq:simplified_encoding}$^T$+\eqref{eq:transpose_graph} & \href{https://oeis.org/A003030}{\texttt{A003030}} \\
        $\mathcal G_{n_0,n}^W$ & weakly connected digraphs & \eqref{eq:general_encoding}+\eqref{eq:general_encoding}$^U$+\eqref{eq:underlying_graph} & \eqref{eq:general_encoding}+\eqref{eq:simplified_encoding}$^U$+\eqref{eq:underlying_graph} & \href{https://oeis.org/A003027}{\texttt{A003027}} \\
        $\mathcal G_{n_0,n}^D$ & DAGs & \eqref{eq:general_encoding}+\eqref{eq:DAG} & N/A & \href{https://oeis.org/A003024}{\texttt{A003024}} \\
        $\mathcal G_{n_0,n}^{D,W}$ & weakly connected DAGs & \eqref{eq:general_encoding}+\eqref{eq:DAG}+\eqref{eq:general_encoding}$^U$+\eqref{eq:underlying_graph} & \eqref{eq:general_encoding}+\eqref{eq:DAG}+\eqref{eq:simplified_encoding}$^U$+\eqref{eq:underlying_graph} & \href{https://oeis.org/A082402}{\texttt{A082402}} \\
        $\mathcal G_{n_0,n}^{D,st}$ & st-DAGs & \eqref{eq:general_encoding}+\eqref{eq:DAG}+\eqref{eq:st-DAG} & N/A & \href{https://oeis.org/A165950}{\texttt{A165950}}\\
        \bottomrule
    \end{tabular}
\end{table}

\section{General Graph Encoding}\label{sec:general_encoding}

This section encodes graphs using IP. We ignore node/edge features and focus on graph structures. To avoid graph isomorphism caused by node indexing, assume that all nodes are labeled differently. Intuitively, encoding such a general graph space is easy, since each graph is uniquely determined by its adjacency matrix, and one only needs to define the $n\times n$ adjacency matrix containing binary variable $A_{u,v}$ that denotes the existence of edge $u\to v$. However, this naive encoding has no extra graph information, e.g., connectivity, reachability, and shortest distance. Encoding these graph properties into decision space is significantly more challenging because we must define constraints that prescribe all variables to have correct values for any possible graph in the search space.

To begin with, we define variables corresponding to relevant graph properties in Table \ref{tab:variables}. We consider all graphs with node number ranging from $n_0$ to $n$. For simplicity, we use $[n]$ to denote the set $\{0,1,\dots, n-1\}$. For each variable $Var$ in Table \ref{tab:variables}, we use $Var(G)$ to denote its value on a given graph $G$. For example, $d_{u,v}(G)$ is the shortest distance from node $u$ to node $v$ in graph $G$. 

\begin{table}[]
    \centering
    \caption{Variables introduced to encode graphs with at most $n$ nodes. Since the shortest distance between two nodes is always less than $n$, we use $n$ to denote infinity, i.e., $d_{u,v}=n$ means node $u$ cannot reach node $v$.}
    \label{tab:variables}
    \begin{tabular}{ccc}
        \toprule
         Variables & Domain & Description \\  
        \midrule
        $A_{v,v},~v\in [n]$ & $\{0,1\}$ & if node $v$ exists\\
        $A_{u,v},~u,v\in [n],~u\neq v$ & $\{0,1\}$ &  if edge $u\to v$ exists \\
        $r_{u,v},~u,v\in [n]$ & $\{0,1\}$ & if node $u$ can reach node $v$\\
        $d_{u,v},~u,v\in [n]$ & $[n+1]$ &  the shortest distance from node $u$ to node $v$ \\
        $\delta_{u,v}^w,~u,v,w\in [n]$ & $\{0,1\}$ & if node $w$ appears on the shortest path from node $u$ to node $v$\\
        \bottomrule
    \end{tabular}    
\end{table}

If graph $G$ is given, all variable values can be easily obtained using classic shortest-path algorithms like the Floyd–Warshall algorithm \citep{floyd1962algorithm}. However, for graph optimization, those variables need to be constrained properly so that they have correct values for any given graph. We first provide a list of necessary conditions that these variables should satisfy based on their definitions. Then we prove that these conditions are sufficient.

\textbf{Condition ($\mathcal C_1$):} At least $n_0$ nodes exist. W.l.o.g., assume that nodes with smaller indexes exist, as shown in Eq.~\eqref{eq:C1}.

\textbf{Condition ($\mathcal C_2$):} Initialization for nonexistent nodes, i.e., if either node $u$ or node $v$ does not exist, edge $u\to v$ cannot exists, node $u$ cannot reach node $v$, and the shortest distance from node $u$ to node $v$ is infinity, i.e., $n$:
\begin{equation*}
    \begin{aligned}
        \min(A_{u,u},A_{v,v})=0\Rightarrow A_{u,v}=0,~r_{u,v}=0,~d_{u,v}=n,~\forall u,v\in [n],~u\neq v
    \end{aligned}
\end{equation*}
which could be rewritten as Eq.~\eqref{eq:C2} with $n$ as a big-M coefficient.

\textbf{Condition ($\mathcal C_3$):} Eq.~\eqref{eq:C3} is the initialization for single node, i.e., node $v$ can reach itself with shortest distance as $0$, and node $v$ is obviously the only node that appears on the shortest path from node $v$ to itself.

\textbf{Condition ($\mathcal C_4$):} Initialization for each edge, i.e., if edge $u\to v$ exists, node $u$ can reach node $v$ with shortest distance as $1$. Otherwise, the shortest distance from node $u$ to node $v$ is larger than $1$:
\begin{equation*}
    \begin{aligned}
        A_{u,v}=1&\Rightarrow r_{u,v}=1,~d_{u,v}=1,&&\forall u,v\in [n],~u\neq v\\
        A_{u,v}=0&\Rightarrow d_{u,v}>1,&&\forall u,v\in [n],~u\neq v
    \end{aligned}
\end{equation*}
which could be rewritten as Eq.~\eqref{eq:C4} with $n-1$ as a big-M coefficient.

\textbf{Condition ($\mathcal C_5$):} Compatibility between distance and reachability, i.e,, node $u$ can reach node $v$ if and only the shortest distance from node $u$ to node $v$ is finite:
\begin{equation*}
    \begin{aligned}
        d_{u,v}<n\Leftrightarrow r_{u,v}=1,~\forall u,v\in[n],~u\neq v
    \end{aligned}
\end{equation*}
which could be rewritten as Eq.~\eqref{eq:C5} with $n-1$ as a big-M coefficient.

\textbf{Condition ($\mathcal C_6$):} Compatibility between path and reachability, i.e., if node $w$ appears on the shortest path from node $u$ to node $v$, then node $u$ can reach node $w$, and node $w$ can reach node $v$ (the opposite is not always true), which means that node $u$ can reach node $v$ via node $w$:
\begin{equation*}
    \begin{aligned}
        \delta_{u,v}^w=1\Rightarrow r_{u,w}=r_{w,v}=1\Rightarrow r_{u,v}=1,\forall u,v,w\in [n],~u\neq v\neq w
    \end{aligned}
\end{equation*}
which is equivalent to Eq.~\eqref{eq:C6}.

\textbf{Condition ($\mathcal C_7$):} Construction of shortest path, i.e., (i) always assume that both node $u$ and node $v$ appear on the shortest path from node $u$ to node $v$ for well-definedness, (ii) if edge $u\to v$ exists or node $u$ cannot reach node $v$, then no other nodes can appear on the shortest path from node $u$ to node $v$, (iii) if edge $u\to v$ does not exist but node $u$ can reach node $v$, then at least one node except for node $u$ and node $v$ will appear on the shortest path from node $u$ to node $v$:
\begin{equation*}
    \begin{aligned}
        \delta_{u,v}^u&=\delta_{u,v}^v=1,&&\forall u,v\in [n],~u\neq v\\
        A_{u,v}=1\lor r_{u,v}=0&\Rightarrow \sum\limits_{w\in [n]}\delta_{u,v}^w=2,&&\forall u,v\in [n],~u\neq v\\
        A_{u,v}=0\land r_{u,v}=1&\Rightarrow \sum\limits_{w\in [n]}\delta_{u,v}^w>2,&&\forall u,v\in [n],~u\neq v
    \end{aligned}
\end{equation*}
Observing that $A_{u,v}=1\lor r_{u,v}=0\Leftrightarrow r_{u,v}-A_{u,v}=0$ since $r_{u,v}\ge A_{u,v}$ always holds, we can rewrite these constraints as Eq.~\eqref{eq:C7} with $n-2$ as a big-M coefficient.

\textbf{Condition ($\mathcal C_8$):} Triangle inequality of shortest distance, i.e., if node $u$ can reach node $w$ and node $w$ can reach node $v$, then the shortest distance from node $u$ to node $v$ is no larger than the shortest distance from node $u$ to node $w$ then to node $v$, and the equality holds when node $w$ appears on the shortest path from node $u$ to node $v$:
\begin{equation*}
    \begin{aligned}
        \delta_{u,v}^w=1&\Rightarrow d_{u,v}=d_{u,w}+d_{w,v},&&\forall u,v,w\in [n],~u\neq v\neq w\\
        r_{u,w}=r_{w,v}=1\land \delta_{u,v}^w=0&\Rightarrow d_{u,v}<d_{u,w}+d_{w,v},&&\forall u,v,w\in [n],~u\neq v\neq w
    \end{aligned}
\end{equation*}
where we omit $r_{u,w}=r_{w,v}=1$ in the first line since $\delta_{u,v}^w=1$ implies it. Similarly, we can rewrite these constraints as Eq.~\eqref{eq:C8} with $n+1$ and $2n$ as big-M coefficients.

Putting Conditions ($\mathcal C_1$)--($\mathcal C_8$) together presents the final formulation Eq.~\eqref{eq:general_encoding}, which comprises linear constraints resulting from these conditions. For simplicity, we omit the domain for indexes in Eq.~\eqref{eq:general_encoding}, i.e., $u,v,w$ are arbitrarily chosen from $[n]$ without duplicates.
\begin{subequations}\label{eq:general_encoding}
    \begin{align}
        &\sum\limits_{v\in [n]}A_{v,v}\ge n_0,~A_{v,v}\ge A_{v+1,v+1}\label{eq:C1}\\
        &2\cdot r_{u,v}\le A_{u,u}+A_{v,v},~d_{u,v}\ge n\cdot (1-A_{u,u}),~d_{u,v}\ge n\cdot (1-A_{v,v})\label{eq:C2}\\
        &r_{v,v}=1,~d_{v,v}=0,~\delta_{v,v}^v=1,~\delta_{v,v}^w=0\label{eq:C3}\\
        &r_{u,v}\ge A_{u,v},
        ~2-A_{u,v}\le d_{u,v}\le 1+(n-1)\cdot (1-A_{u,v})\label{eq:C4}\\
        &r_{u,v}\le n-d_{u,v}\le (n-1)\cdot r_{u,v}\label{eq:C5}\\
        &2\cdot \delta_{u,v}^w\le r_{u,w}+r_{w,v}\le 1+r_{u,v}\label{eq:C6}\\
        &\delta_{u,v}^u=\delta_{u,v}^v=1,~2+r_{u,v}-A_{u,v}\le \sum\limits_{w\in [n]}\delta_{u,v}^w\le 2+(n-2)\cdot(r_{u,v}-A_{u,v})\label{eq:C7}\\
        &(1-\delta_{u,v}^w)-(n+1)\cdot (2-r_{u,w}-r_{w,v})\le d_{u,w}+d_{w,v}-d_{u,v}\le 2n\cdot (1-\delta_{u,v}^w)\label{eq:C8}
    \end{align}
\end{subequations}

Constraints for optimization formulations must be carefully selected. There are often multiple ways to encode a combinatorial problem, but insufficient constraints result in an unnecessarily large search space with symmetric solutions, while excessive constraints may cutoff feasible solutions from the search space. Lemma \ref{lm:existence} shows that all constraints in Eq.~\eqref{eq:general_encoding} are necessary conditions.

\begin{lemma}\label{lm:existence}
    Given any graph $G$ with $n$ nodes, $\{A_{u,v}(G),r_{u,v}(G),d_{u,v}(G),\delta_{u,v}^w(G)\}_{u,v,w\in [n]}$ is a feasible solution of Eq.~\eqref{eq:general_encoding} with $n_0=n$.
\end{lemma}
\begin{proof}
    By definition, it is easy to check that $\{A_{u,v}(G),r_{u,v}(G),d_{u,v}(G),\delta_{u,v}^w(G)\}_{u,v,w\in [n]}$ satisfies Conditions ($\mathcal C_1$)--($\mathcal C_8$). 
\end{proof}

The opposite is non-trivial to prove, that is, any feasible solution of Eq.~\eqref{eq:general_encoding} corresponds to an unique graph with $\sum_{v\in [n]}A_{v,v}$ nodes, which is guaranteed by Theorem \ref{thm:uniqueness}.

\begin{theorem}\label{thm:uniqueness}
    There is a bijection between the feasible domain restricted by Eq.~\eqref{eq:general_encoding} and the whole graph space with node numbers ranging from $n_0$ to $n$.
\end{theorem}
\begin{proof}
    Denote $\mathcal F_{n_0,n}$ as the feasible domain restricted by Eq.~\eqref{eq:general_encoding}, and $\mathcal G_{n_0,n}$ as the whole graph space with node numbers ranging from $n_0$ to $n$. Define the following mapping:
    \begin{equation*}
        \begin{aligned}
            \mathcal M_{n_0,n}:\mathcal F_{n_0,n}&\to\mathcal G_{n_0,n}\\
            \{A_{u,v},r_{u,v},d_{u,v},\delta_{u,v}^w\}_{u,v,w\in [n]}&\mapsto \{A_{u,v}\}_{u,v\in [n]}
        \end{aligned}
    \end{equation*}
    For simplicity, we still use a $n\times n$ adjacency matrix to define a graph with less than $n$ nodes, and use $A_{v,v}(G)$ to represent the existence of node $v$. Subscriptions, e.g., $\{\}_{u,v,w\in [n]}$, are omitted. 
    
    If $n_1=\sum_{v\in [n]}A_{v,v}<n$, Condition $(\mathcal C_1)$ forces that:
    \begin{equation*}
        \begin{aligned}
            A_{v,v}=
            \begin{cases}
            1,&v\in [n_1]\\
            0,&v\in [n]\backslash [n_1]
            \end{cases}
        \end{aligned}
    \end{equation*}
    For any pair of $(u,v)$ with $u\neq v$ and $\max(u,v)\ge n_1$, Conditions $(\mathcal C_2)$ and $(\mathcal C_7)$ uniquely define $\{r_{u,v},d_{u,v},\delta_{u,v}^w\}$ as:
    \begin{equation*}
        \begin{aligned}
            r_{u,v}=0,~d_{u,v}=n,~\delta_{u,v}^w=\begin{cases}
                1,&w\in  \{u,v\}\\
                0,&w\not\in \{u,v\}
            \end{cases}
        \end{aligned}
    \end{equation*}
    It is equivalent to show that $\mathcal M_{n,n}$ is a bijection. Since Lemma \ref{lm:existence} shows that $\mathcal M_{n,n}$ is a surjection, it suffices to prove that $\mathcal M_{n,n}$ is an injection. Precisely, for any feasible solution $\{A_{u,v},r_{u,v},d_{u,v},\delta_{u,v}^w\}$, there exists a graph $G$ with adjacency matrix given by $\{A_{u,v}(G)\}=\{A_{u,v}\}$, such that:
    \begin{equation}\label{eq:bijection}
        \begin{aligned}
            \{r_{u,v}(G),d_{u,v}(G),\delta_{u,v}^w(G)\}=\{r_{u,v},d_{u,v},\delta_{u,v}^w\}
        \end{aligned}
    \end{equation}
    Since $r_{v,v}(G), d_{v,v}(G),\delta_{v,v}^w(G),\delta_{u,v}^u(G),\delta_{u,v}^v(G)$ are defined for completeness, whose variable counterparts are properly and uniquely defined in Condition $(\mathcal C_3)$ and Condition $(\mathcal C_7)$, we only need to consider all triples $(u,v,w)$ with $u\neq v\neq w$, which will not be specified later for simplicity.

    Now we are going to prove Eq.~\eqref{eq:bijection} holds by induction on $\min(d_{u,v}(G),d_{u,v})<n$.

    When $\min(d_{u,v}(G),d_{u,v})=1$, for any pair of $(u,v)$, we have:
    \begin{equation*}
        \begin{aligned}
            d_{u,v}(G)=1 &\Rightarrow A_{u,v}(G)=1,~r_{u,v}(G)=1,~\delta_{u,v}^w(G)=0 &&\longleftarrow \text{definition}\\
            &\Rightarrow A_{u,v}=1 &&\longleftarrow \text{definition of $\mathcal M_{n,n}$}\\
            &\Rightarrow r_{u,v}=1,~d_{u,v}=1,~\delta_{u,v}^w=0 &&\longleftarrow \text{Conditions $(\mathcal C_4) + (\mathcal C_7)$}
        \end{aligned}
    \end{equation*}
    and:
    \begin{equation*}
        \begin{aligned}
            d_{u,v}=1 &\Rightarrow A_{u,v}=1,~r_{u,v}=1,~\delta_{u,v}^w=0 &&\longleftarrow \text{Conditions $(\mathcal C_4) + (\mathcal C_7)$}\\
            &\Rightarrow A_{u,v}(G)=1 &&\longleftarrow \text{definition of $\mathcal M_{n,n}$}\\
            &\Rightarrow r_{u,v}(G)=1,~d_{u,v}(G)=1,~\delta_{u,v}^w(G)=0 &&\longleftarrow \text{definition}
        \end{aligned}
    \end{equation*}
    For both cases, Eq.~\eqref{eq:bijection} holds.

    Assume that Eq.~\eqref{eq:bijection} holds for any pair of $(u,v)$ with $\min(d_{u,v}(G),d_{u,v})\le sd$ with $sd<n-1$. Consider the following two cases for $\min(d_{u,v}(G),d_{u,v})=sd+1<n$.

    \textbf{Case I:} If $d_{u,v}(G)=sd+1$, we know that $r_{u,v}(G)=1$ since the shortest distance from node $u$ to node $v$ is finite. For any $w\not\in\{u,v\}$ such that $\delta_{u,v}^w(G)=1$, we have:
    \begin{equation*}
        \begin{aligned}
            \delta_{u,v}^w(G)=1 &\Rightarrow d_{u,w}(G)+d_{w,v}(G)=d_{u,v}(G)&&\longleftarrow\text{definition of $\delta_{u,v}^w(G)$} \\
             &\Rightarrow \max(d_{u,w}(G),d_{w,v}(G))\le sd &&\longleftarrow \text{$d_{u,w}(G)>0, d_{w,v}(G)>0$}\\
             &\Rightarrow d_{u,w}=d_{u,w}(G),~d_{w,v}=d_{w,v}(G) &&\longleftarrow\text{assumption of induction}\\
             & \Rightarrow r_{u,w}=r_{w,v}=1&&\longleftarrow\text{Condition $(\mathcal C_5)$} \\
             &\Rightarrow d_{u,v}\le d_{u,w}+d_{w,v}=sd+1&&\longleftarrow\text{Condition $(\mathcal C_8)$}\\
             &\Rightarrow d_{u,v}=sd+1 &&\longleftarrow\text{$d_{u,v}\ge sd+1$}\\
             &\Rightarrow r_{u,v}=1,~\delta_{u,v}^w=1&&\longleftarrow\text{Conditions $(\mathcal C_5)+(\mathcal C_8)$}
        \end{aligned}
    \end{equation*}
    which means that $r_{u,v}=r_{u,v}(G),~d_{u,v}=d_{u,v}(G),~\delta_{u,v}^w=\delta_{u,v}^w(G)$ with $\delta_{u,v}^w(G)=1$.
    
    For any $w\not\in\{u,v\}$ such that $\delta_{u,v}^w(G)=0$. If $\delta_{u,v}^w=1$, then we have:
    \begin{equation*}
        \begin{aligned}
            \delta_{u,v}^w=1&\Rightarrow r_{u,w}=r_{w,v}=1,~d_{u,v}=d_{u,w}+d_{w,v}&&\longleftarrow\text{Conditions $(\mathcal C_6)+(\mathcal C_8)$}\\
            &\Rightarrow \max(d_{u,w},d_{w,v})\le sd &&\longleftarrow\text{$d_{u,w}>0,~d_{w,v}>0$}\\
            &\Rightarrow d_{u,w}(G)=d_{u,w},~d_{w,v}(G)=d_{w,v}&&\longleftarrow\text{assumption of induction}\\
            &\Rightarrow d_{u,w}(G)+d_{w,v}(G)=sd+1=d_{u,v}(G)&&\longleftarrow\text{$d_{u,v}(G)=d_{u,v}=sd+1$}\\
            &\Rightarrow\delta_{u,v}^w(G)=1&&\longleftarrow\text{definition of $\delta_{u,v}^w(G)$}
        \end{aligned}
    \end{equation*}
    which contradicts $\delta_{u,v}^w(G)=0$. Thus $\delta_{u,v}^w=0=\delta_{u,v}^w(G)$ with $\delta_{u,v}^w(G)=0$.

    \textbf{Case II:} If $d_{u,v}=sd+1$, from $sd+1>1$ and Condition $(\mathcal C_4)$ we know that $A_{u,v}=0$, from Condition $(\mathcal C_5)$ we have $r_{u,v}=1$, and then from Condition $(\mathcal C_7)$ we obtain that $\sum_{w\in [n]}\delta_{u,v}^w>2$. 

    For any $w\not\in\{u,v\}$ such that $\delta_{u,v}^w=1$, we have:
    \begin{equation*}
        \begin{aligned}
            \delta_{u,v}^w=1 &\Rightarrow d_{u,w}+d_{w,v}=d_{u,v}=sd+1&&\longleftarrow\text{Condition $(\mathcal C_8)$} \\
             &\Rightarrow \max(d_{u,w},d_{w,v})\le sd &&\longleftarrow \text{$d_{u,w}>0, d_{w,v}>0$}\\
             &\Rightarrow d_{u,w}(G)=d_{u,w},~d_{w,v}(G)=d_{w,v} &&\longleftarrow\text{assumption of induction}\\
             &\Rightarrow d_{u,v}(G)\le d_{u,w}(G)+d_{w,v}(G)=sd+1&&\longleftarrow\text{definition of $d_{u,v}(G)$}\\
             &\Rightarrow d_{u,v}(G)=sd+1 &&\longleftarrow\text{$d_{u,v}(G)\ge sd+1$}\\
             &\Rightarrow r_{u,v}(G)=1,~\delta_{u,v}^w(G)=1&&\longleftarrow\text{definition}
        \end{aligned}
    \end{equation*}
    which means that $r_{u,v}(G)=r_{u,v},~d_{u,v}(G)=d_{u,v},~\delta_{u,v}^w(G)=\delta_{u,v}^w$ with $\delta_{u,v}^w=1$.

    For any $w\not\in\{u,v\}$ such that $\delta_{u,v}^w=0$. If $\delta_{u,v}^w(G)=1$, then we have:
    \begin{equation*}
        \begin{aligned}
            \delta_{u,v}^w(G)=1&\Rightarrow d_{u,v}(G)=d_{u,w}(G)+d_{w,v}(G)&&\longleftarrow\text{definition of $\delta_{u,v}^w(G)$}\\
            &\Rightarrow \max(d_{u,w}(G),d_{w,v}(G))\le sd &&\longleftarrow\text{$d_{u,w}(G)>0,~d_{w,v}(G)>0$}\\
            &\Rightarrow d_{u,w}=d_{u,w}(G),~d_{w,v}=d_{w,v}(G)&&\longleftarrow\text{assumption of induction}\\
            &\Rightarrow d_{u,w}+d_{w,v}=sd+1=d_{u,v}&&\longleftarrow\text{$d_{u,v}(G)=d_{u,v}=sd+1$}\\
            &\Rightarrow\delta_{u,v}^w=1&&\longleftarrow\text{Condition $(\mathcal C_8)$}
        \end{aligned}
    \end{equation*}
    which contradicts $\delta_{u,v}^w=0$. Thus $\delta_{u,v}^w(G)=0=\delta_{u,v}^w$ with $\delta_{u,v}^w=0$.

    The remaining case is $d_{u,v}(G)=d_{u,v}=n$, i.e., node $u$ cannot reach node $v$, which is verified by:
    \begin{equation*}
        \begin{aligned}
            r_{u,v}&=0=r_{u,v}(G)&&\longleftarrow\text{Condition $(\mathcal C_5)$, definition of $r_{u,v}(G)$}\\
            \delta_{u,v}^w&=0=\delta_{u,v}^w(G)&&\longleftarrow\text{Condition $(\mathcal C_7)$, definition of $\delta_{u,v}^w(G)$}
        \end{aligned}
    \end{equation*}
    
    Therefore, Eq.~\eqref{eq:bijection} always holds, which completes the proof.
\end{proof}

Graph encoding Eq.~\eqref{eq:general_encoding} does not make assumptions over the graph space, showing its generality and flexibility to various graph-based tasks. Meanwhile, Eq.~\eqref{eq:general_encoding} may be easily restricted to specific graph types. For instance,

\textbf{Undirected graphs:} Add symmetry constraints to get undirected graphs:
\begin{equation}\label{eq:undirect}
    \begin{aligned}
    A_{u,v}=A_{v,u},~r_{u,v}=r_{v,u},~d_{u,v}=d_{v,u},~\delta_{u,v}^w=\delta_{v,u}^w,~\forall u,v,w\in [n],~u<v
    \end{aligned}
\end{equation}

\begin{remark}
$A_{u,v}=A_{v,u}$ is enough for undirected graphs, while the rest constraints are introduced to tighten the continuous relaxation of our encoding.
\end{remark}

\textbf{Connected undirected graphs \& Strongly connected digraphs:} Each existing node can reach all other existing nodes, i.e.,
\begin{equation*}
    \begin{aligned}
        A_{u,u}=A_{v,v}=1\Rightarrow r_{u,v}=1,~\forall u,v\in [n],~u\neq v
    \end{aligned}
\end{equation*}
which can be equivalently rewritten as the following constraint:
\begin{equation}\label{eq:strong_connectivity}  
    \begin{aligned}
        r_{u,v}\ge A_{u,u}+A_{v,v}-1,~\forall u,v\in [n],~u\neq v
    \end{aligned}
\end{equation}

\textbf{Acyclic digraphs (DAGs):} No pair of nodes can reach each other, i.e.,
\begin{equation}\label{eq:DAG}
    \begin{aligned}
        r_{u,v}+r_{v,u}\le 1,~\forall u,v\in [n],~u<v
    \end{aligned}
\end{equation}

\textbf{DAGs with one source and one sink (st-DAGs):} If two existing nodes have no in-neighbors (or out-neighbors, respectively), i.e., zero in-degree (out-degree, respectively), then they are identical:
\begin{equation*}
    \begin{aligned}
        \sum\limits_{w\neq u}A_{w,u}=0\land \sum\limits_{w\neq v}A_{w,v}=0\land A_{u,u}=A_{v,v}=1&\Rightarrow u=v,~\forall u,v\in [n]\\
        \sum\limits_{w\neq u}A_{u,w}=0\land \sum\limits_{w\neq v}A_{v,w}=0\land A_{u,u}=A_{v,v}=1&\Rightarrow u=v,~\forall u,v\in [n]
    \end{aligned}
\end{equation*}
which can be linearly represented as:
\begin{equation}\label{eq:st-DAG}
    \begin{aligned}
        &\sum\limits_{w\neq u}A_{w,u}+\sum\limits_{w\neq v}A_{w,v}\ge A_{u,u}+A_{v,v}-1,~\forall u,v\in [n],~u\neq v\\
        &\sum\limits_{w\neq u}A_{u,w}+\sum\limits_{w\neq v}A_{v,w}\ge A_{u,u}+A_{v,v}-1,~\forall u,v\in [n],~u\neq v
    \end{aligned}
\end{equation}

st-DAGs are special cases of weakly connected graphs, i.e., graphs with connected underlying graphs. Weak connectivity is very hard to formulate if one purely uses variables in Table \ref{tab:variables}, since it is not defined over the original graph but its underlying graph, i.e., removing all directions of edges. A straightforward way is applying Eqs.~\eqref{eq:general_encoding} and \eqref{eq:strong_connectivity} to ensure the connectivity of the underlying graphs. For simplicity, we do not write the full encoding here since it is just repeating Eq.~\eqref{eq:general_encoding} twice (once for the original graph space, once for the underlying graph space), as well as the following constraints linking graph $G$ with its underlying graph $G^U$:
\begin{equation}\label{eq:underlying_graph}
    \begin{aligned}
        A^U_{v,v}=A_{v,v},~A_{u,v}+A_{v,u}\le 2\cdot A^U_{u,v} \le 2\cdot (A_{u,v}+A_{v,u}),~\forall u,v\in [n],~u\neq v
    \end{aligned}
\end{equation}
where $A^U$ is the adjacency matrix of $G^U$, and the last constraint is rewritten from:
\begin{equation*}
    \begin{aligned}
        A^U_{u,v}=\max(A_{u,v},A_{v,u}),~\forall u,v\in [n],~u\neq v
    \end{aligned}
\end{equation*}

\section{Simplified Encoding using Connectivity}\label{sec:simplified_encoding}
Graph encoding proposed in Section \ref{sec:general_encoding} introduces $O(n^3)$ binary variables and $O(n^3)$ constraints to handle multiple strongly connected components and maintain the shortest paths between any pair of nodes. This section shows that how to simplify Eq.~\eqref{eq:general_encoding} given strong connectivity, i.e., using only $O(n^2)$ binary variables and $O(n^2)$ constraints. The cost is losing the shortest paths between any pair of nodes.  

Strong connectivity means that $A_{u,u}=A_{v,v}=1\Rightarrow r_{u,v}=1,~\forall u,v\in [n]$. Alternatively, it could be achieved by forcing $A_{v,v}=1\Rightarrow r_{s,v}=r_{v,s}=1,~\forall v\in [n]$, where $s$ is any existed node, i.e., $A_{s,s}=1$. Note that $r_{v,s}=1$ over $G$ is equivalent to $r_{s,v}=1$ in the transpose of $G$, i.e., reversing all directions of edges. W.l.o.g., assume $s=0$ since node $0$ always exists due to Condition $(\mathcal C_1)$. Precisely, denote $\mathcal G_{n_0,n}^{\to}\subset\mathcal G_{n_0,n}$ as the set of graphs  where node $0$ can reach all existed nodes, and $\mathcal G_{n_0,n}^{\gets}\subset\mathcal G_{n_0,n}$ as the set of graphs that all existed nodes can reach node $0$. Then the set of strongly connected graphs, denoted by $\mathcal G_{n_0,n}^S$, is the intersection of $\mathcal G_{n_0,n}^{\to}$ and $\mathcal G_{n_0,n}^{\gets}$.

\begin{remark}
    Observe that graphs in $\mathcal G_{n_0,n}^{\gets}$ are transpose of graphs in $\mathcal G_{n_0,n}^{\to}$. If we could encode graphs in $\mathcal G_{n_0,n}^{\to}$, then encoding $\mathcal G_{n_0,n}^{\gets}$ is simply done by transposing the adjacency matrix $\{A_{u,v}\}_{\forall u,v\in [n]}$. For simplicity, we only present the encoding of $\mathcal G_{n_0,n}^{\to}$.
\end{remark}

Instead of considering all shortest paths, we only maintain the shortest paths from node $0$ to node $v\in V$, which could be updated by combining a shortest path from node $0$ to node $u$ and edge $u\to v$. Our idea is quite similar to using breadth first search (BFS) to calculate single-source shortest paths. The difference is that BFS is conducted over a given graph, while our constraints are used to find all graphs and obtain the corresponding shortest paths simultaneously.

\begin{table}[]
    \centering
    \caption{Variables introduced to encode $\mathcal G_{n_0,n}^{\to}$.
    Since the shortest distance from node $0$ to any existed node is always less than $n$, we use $d_{v}=n$ to denote that node $v$ does not exist.}
    \label{tab:variables_simplified}
    \begin{tabular}{ccc}
        \toprule
         Variables & Domain & Description \\  
        \midrule
        $A_{v,v},~v\in [n]$ & $\{0,1\}$ & if node $v$ exists\\
        $A_{u,v},~u,v\in [n],~u\neq v$ & $\{0,1\}$ &  if edge $u\to v$ exists \\
        $d_{v},~v\in [n]$ & $[n+1]$ &  the shortest distance from node $0$ to node $v$ \\
        $\gamma_{u,v},~u,v\in [n]$ & $\{0,1\}$ & if $d_v=d_u+1$ and $A_{u,v}=1$\\
        \bottomrule
    \end{tabular}
\end{table}

Table \ref{tab:variables_simplified} lists involved variables in our encoding. As before, we first list necessary conditions that those variables should satisfy, then prove that these conditions are sufficient. 

\textbf{Condition $(\mathcal C^\to_1)$: } Same to Condition $(\mathcal C_1)$. We repeat it here and in Eq.~\eqref{eq:CS1} for completeness.

\textbf{Condition $(\mathcal C^\to_2)$: } Existence of each edge, i.e., edge $u\to v$ does not exist if either node $u$ or node $v$ does not exist. Also, $\gamma_{u,v}=1$ implies that edge $u\to v$ exists by definition, as shown in Eq.~\eqref{eq:CS2}.

\textbf{Condition $(\mathcal C^\to_3)$: } Construction of shortest path, i.e., if node $v\neq 0$ exists, then it has at least one in-neighbor $u$ appearing on the shortest path from node $0$ to node $v$. Otherwise, we need to properly define $\gamma_{u,v}$ for node $0$ and nonexistent nodes:
\begin{equation*}
    \begin{aligned}
        A_{v,v}=1\land v\neq 0&~\Rightarrow \sum\limits_{u\in [n]}\gamma_{u,v}\ge 1,&&\forall v\in [n]\\
        A_{v,v}=0\lor v=0&~\Rightarrow \sum\limits_{u\in [n]}\gamma_{u,v}=0,&&\forall v\in [n]
    \end{aligned}
\end{equation*}
which could be equivalently represented as Eq.~\eqref{eq:CS3} with $n-1$ as a big-M coefficient. Note that we also initialize $\gamma_{v,v}=0$ in Eq.~\eqref{eq:CS3} for well-definedness.

\textbf{Condition $(\mathcal C^\to_4)$: } Initialization of shortest distance, i.e., the shortest distance from node $0$ to itself is $0$, to an existed node ranges from $1$ to $n-1$, and to a nonexistent node is infinity, i.e., $n$:
\begin{equation*}
    \begin{aligned}
        A_{v,v}=1&~\Rightarrow 1\le d_v<n,&&\forall v\in [n]\backslash\{0\}\\
        A_{v,v}=0&~\Rightarrow d_v=n, && \forall v\in [n]\backslash\{0\}
    \end{aligned}
\end{equation*}
which could be linearly expressed as Eq.~\eqref{eq:CS4} with $n-1$ as a big-M coefficient.

\textbf{Condition $(\mathcal C^\to_5)$: } Update of shortest distance, i.e., for any in-neighbor $u$ of node $v$, the shortest distance from node $0$ to node $v$ is no more than the shortest distance from node $0$ to node $u$ plus one. Furthermore, the equality holds if node $u$ appears on the shortest path from node $0$ to node $v$:
\begin{equation*}
    \begin{aligned}
        A_{u,v}=1\land \gamma_{u,v}=1&~\Rightarrow d_v=d_u+1,&&\forall u,v\in [n],~u\neq v\\
        A_{u,v}=1\land \gamma_{u,v}=0&~\Rightarrow d_v<d_u+1,&&\forall u,v\in [n],~u\neq v
    \end{aligned}
\end{equation*}
Similar to Eq.~\eqref{eq:C8}, we represent this condition as Eq.~\eqref{eq:CS5} with $n$ and $n+1$ as big-M coefficients.

Putting Conditions $(\mathcal C^\to_1)$--$(\mathcal C^\to_5)$ together yields the following encoding:
\begin{subequations}\label{eq:simplified_encoding}
    \begin{align}
        & \sum\limits_{v\in [n]}A_{v,v}\ge n_0,~A_{v,v}\ge A_{v+1,v+1}&&\forall v\in [n-1]\label{eq:CS1}\\
        & 2\cdot A_{u,v}\le A_{u,u}+A_{v,v},~A_{u,v}\ge \gamma_{u,v}&&\forall u,v\in [n],~u\neq v\label{eq:CS2}\\
        & \sum\limits_{u\in [n]}\gamma_{u,0}=0,~\gamma_{v,v}=0,~A_{v,v}\le \sum\limits_{u\in [n]}\gamma_{u,v}\le (n-1)\cdot A_{v,v}&&\forall v\in [n]\backslash\{0\}\label{eq:CS3}\\
        & d_0=0,~1+(n-1)\cdot (1-A_{v,v})\le d_v\le n-A_{v,v}&&\forall v\in [n]\backslash\{0\}\label{eq:CS4}\\
        & (1-\gamma_{u,v})-n\cdot (1-A_{u,v})\le d_u+1-d_v\le (n+1)\cdot(1-\gamma_{u,v})&&\forall u,v\in [n],~u\neq v\label{eq:CS5}
    \end{align}   
\end{subequations}

\begin{theorem}\label{thm:bijection_simplified}
    There is a bijection between the feasible domain restricted by Eq.~\eqref{eq:simplified_encoding}  and graph space $\mathcal G_{n_0,n}^{\to}$.
\end{theorem}
\begin{proof}
    Denote $\mathcal F_{n_0,n}^{\to}$ as the feasible domain restricted by Eq.~\eqref{eq:simplified_encoding}. Define the following mapping:
    \begin{equation*}
        \begin{aligned}
            \mathcal M_{n_0,n}^{\to}:\mathcal F_{n_0,n}^{\to}&~\to \mathcal G_{n_0,n}^{\to}\\
            \{A_{u,v},d_v,\gamma_{u,v}\}_{u,v\in [n]}&~\mapsto \{A_{u,v}\}_{u,v\in [n]}
        \end{aligned}  
    \end{equation*}
    All variables involving nonexistent nodes are properly and uniquely defined, and Conditions $(\mathcal C^\to_1)$--$(\mathcal C^\to_5)$ are necessary conditions that any graph in $\mathcal G_{n_0,n}^{\to}$ satisfies. Following a similar statement as the proof of Theorem \ref{thm:uniqueness}, it suffices to show that $\mathcal M_{n,n}^\to$ is an injection, i.e., for any feasible solution $\{A_{u,v},d_v,\gamma_{u,v}\}_{u,v\in [n]}$, there exists a graph $G\in\mathcal G_{n_0,n}^\to$ with adjacency matrix given by $\{A_{u,v}(G)\}=\{A_{u,v}\}$ such that:
    \begin{equation}\label{eq:bijection_simplified}
        \begin{aligned}
            \{d_v(G),\gamma_{u,v}(G)\}=\{d_v,\gamma_{u,v}\}
        \end{aligned}
    \end{equation}
    Since $d_0,\gamma_{u,0},\gamma_{v,v},~\forall u,v\in [n]$ are well-defined by Conditions $(\mathcal C^\to_3)$ and $(\mathcal C^\to_4)$, we only consider all pairs $(u,v)$ with $u\neq v$ and $v\neq 0$, which will not be specified later.

    Now we are going to prove Eq.~\eqref{eq:bijection_simplified} holds by induction on $1\le \min(d_v(G),d_v)<n$.

    When $\min(d_v(G),d_v)=1$, for any node $v$, we have that:
    \begin{equation*}
        \begin{aligned}
            d_v(G)=1&\Rightarrow A_{0,v}(G)=0,~\gamma_{0,v}(G)=1,~\gamma_{u,v}(G)=0,~\forall u\neq 0 &&\longleftarrow\text{definition}\\
            &\Rightarrow A_{0,v}=1&&\longleftarrow\text{definition of $\mathcal M_{n_0,n}^\to$}\\
            &\Rightarrow d_v=1,~\gamma_{0,v}=1&&\longleftarrow\text{Condition $(\mathcal C^\to_5)$}
        \end{aligned}
    \end{equation*}
    and that:
    \begin{equation*}
        \begin{aligned}
            u\neq 0&\Rightarrow d_u>0 &&\longleftarrow\text{Condition $(\mathcal C^\to_4)$}\\
            &\Rightarrow (n+1)\cdot (1-\gamma_{u,v})\ge d_u+1-d_v=d_u>0&&\longleftarrow\text{Condition $(\mathcal C^\to_5)$}\\
            &\Rightarrow \gamma_{u,v}=0
        \end{aligned}
    \end{equation*}
    On the other hand,
    \begin{equation*}
        \begin{aligned}
            d_v=1&\Rightarrow A_{0,v}=1,~\gamma_{0,v}=1,~\gamma_{u,v}=0,~\forall u\neq 0&&\longleftarrow\text{Condition $(\mathcal C^\to_5)$}\\
            &\Rightarrow A_{0,v}(G)=1&&\longleftarrow\text{definition of $\mathcal M_{n_0,n}^\to$}\\
            &\Rightarrow d_v(G)=1,~\gamma_{0,v}(G)=1,~\gamma_{u,v}(G)=0,~\forall u\neq 0 &&\longleftarrow\text{definition}
        \end{aligned}
    \end{equation*}
    For both cases, Eq.~\eqref{eq:bijection_simplified} holds.

    Assume that Eq.~\eqref{eq:bijection_simplified} holds for any node $v$ with $\min(d_v(G),d_v)\le sd$ with $sd<n-1$. Consider the following two cases for $\min(d_v(G),d_v)=sd+1<n$.

    \textbf{Case I:} If $d_v(G)=sd+1$, for any $u$ such that $\gamma_{u,v}(G)=1$, we have:
    \begin{equation*}
        \begin{aligned}
            \gamma_{u,v}(G)=1&\Rightarrow A_{u,v}(G)=1,~d_u(G)=sd&&\longleftarrow\text{definition of $\gamma_{u,v}$}\\
            &\Rightarrow A_{u,v}=1,~d_u=sd&&\longleftarrow\text{definition of $\mathcal M_{n_0,n}^\to$, assumption of induction}\\
            &\Rightarrow d_v=sd+1,~\gamma_{u,v}=1&&\longleftarrow\text{Condition $(\mathcal C^\to_5)$, $d_v\ge sd+1$}
        \end{aligned}
    \end{equation*}
    For any $u$ such that $\gamma_{u,v}(G)=0$, we either have:
    \begin{equation*}
        \begin{aligned}
            A_{u,v}(G)=0\Rightarrow A_{u,v}=0\Rightarrow \gamma_{u,v}=0&&\longleftarrow\text{definition of $\mathcal M_{n_0,n}^\to$, Condition $(\mathcal C^\to_2)$}
        \end{aligned}
    \end{equation*}
    or $A_{u,v}(G)=1\Rightarrow A_{u,v}=0$. Assuming $\gamma_{u,v}=1$ gives:
    \begin{equation*}
        \begin{aligned}
            \gamma_{u,v}=1&\Rightarrow d_u=sd &&\longleftarrow\text{Condition $(\mathcal C^\to_5)$}\\
            &\Rightarrow d_u(G)=sd&&\longleftarrow\text{assumption of induction}\\
            &\Rightarrow\gamma_{u,v}(G)=1&&\longleftarrow\text{definition of $\gamma_{u,v}(G)$}
        \end{aligned}
    \end{equation*}
    which contradicts $\gamma_{u,v}(G)=0$. Thus $\gamma_{u,v}=0=\gamma_{u,v}(G)$.
    
    \textbf{Case II:} If $d_v=sd+1$, for any $u$ such that $\gamma_{u,v}=1$, we have:
    \begin{equation*}
        \begin{aligned}
            \gamma_{u,v}=1&\Rightarrow A_{u,v}=1,~d_u=sd&&\longleftarrow\text{Condition $(\mathcal C^\to_5)$}\\
            &\Rightarrow A_{u,v}(G)=1,~d_u(G)=sd&&\longleftarrow\text{definition of $\mathcal M_{n_0,n}^\to$, assumption of induction}\\
            &\Rightarrow d_v(G)=sd+1&&\longleftarrow\text{definition of $d_v(G)$, $d_v(G)\ge sd+1$}\\
            &\Rightarrow \gamma_{u,v}(G)=1&&\longleftarrow\text{definition of $\gamma_{u,v}(G)$}
        \end{aligned}
    \end{equation*}
    For any $u$ such that $\gamma_{u,v}=0$, we either have:
    \begin{equation*}
        \begin{aligned}
            A_{u,v}=0\Rightarrow A_{u,v}(G)=0\Rightarrow \gamma_{u,v}(G)=0&&\longleftarrow\text{definition of $\mathcal M_{n_0,n}^\to$ and $\gamma_{u,v}(G)$}
        \end{aligned}
    \end{equation*}
    or $A_{u,v}=1\Rightarrow A_{u,v}(G)=1$. Assuming that $\gamma_{u,v}(G)=1$ gives:
    \begin{equation*}
        \begin{aligned}
            \gamma_{u,v}(G)=1&\Rightarrow d_u(G)=sd&&\longleftarrow\text{definition of $d_u(G)$}\\
            &\Rightarrow d_u=sd&&\longleftarrow\text{assumption of induction}\\
            &\Rightarrow \gamma_{u,v}=1&&\longleftarrow\text{Condition $(\mathcal C^\to_5)$}
        \end{aligned}
    \end{equation*}
    which contradicts $\gamma_{u,v}=0$. Thus $\gamma_{u,v}(G)=0=\gamma_{u,v}$.
\end{proof}

\begin{lemma}\label{lem:bijection_simplified}
    There is a bijection between $\mathcal F_{n_0,n}^S$ and graph space $\mathcal G_{n_0,n}^S$, where:
    \begin{equation*}
        \begin{aligned}
            \mathcal F_{n_0,n}^S=\{\{A_{u,v},d^\to_v,\gamma^\to_{u,v},d^\gets_v,\gamma^\gets_{u,v}\}~|~\{A_{u,v},d^\to_v,\gamma^\to_{u,v}\},\{A_{v,u},d^\gets_v,\gamma^\gets_{u,v}\}\in \mathcal F_{n_0,n}^\to\}
        \end{aligned}
    \end{equation*}
\end{lemma}
\begin{proof}
For any $G\in \mathcal G_{n_0,n}^S=\mathcal G_{n_0,n}^\to\cap\mathcal G_{n_0,n}^\gets$, using Theorem \ref{thm:bijection_simplified} over $G$ and its transpose, there uniquely exists $\{A_{u,v},d^\to_v,\gamma^\to_{u,v}\},\{A_{v,u},d^\gets_v,\gamma^\gets_{u,v}\}\in\mathcal F_{n_0,n}^\to$ such that:
\begin{equation*}
    \begin{aligned}
        \mathcal M_{n_0,n}^\to(\{A_{u,v},d^\to_v,\gamma^\to_{u,v}\})=\{A_{u,v}\},~\mathcal M_{n_0,n}^\to(\{A_{v,u},d^\gets_v,\gamma^\gets_{u,v}\})=\{A_{v,u}\}
    \end{aligned}
\end{equation*}
which gives the following bijection:
\begin{equation*}
    \begin{aligned}
        \mathcal M_{n_0,n}^S:\mathcal F_{n_0,n}^S&~\to \mathcal G_{n_0,n}^S\\
        \{A_{u,v},d^\to_v,\gamma^\to_{u,v},d^\gets_v,\gamma^\gets_{u,v}\}&~\mapsto \{A_{u,v}\}
    \end{aligned}
\end{equation*}
and completes the proof.
\end{proof}

Therefore, to encode $\mathcal G_{n_0,n}^S$, i.e., strongly connected graphs, one needs to repeat Eq.~\eqref{eq:simplified_encoding} twice over graph $G$ and its transpose $G^T$ respectively, as well as the following constraint to link $G$ and $G^T$:
\begin{equation}\label{eq:transpose_graph}
    \begin{aligned}
        A^T_{u,v}=A_{v,u},~\forall u,v\in [n]
    \end{aligned}
\end{equation}
where $A^T$ is the adjacency matrix of $G^T$.

\begin{remark}
    When the graphs are undirected, there is no need to repeat Eq.~\eqref{eq:simplified_encoding} twice over graphs and their transposes. Instead, adding symmetric constraints $A_{u,v}=A_{v,u}$ into Eq.~\eqref{eq:simplified_encoding} is enough since it results in $d_v^\to=d_v^\gets,~\gamma_{u,v}^\to=\gamma_{u,v}^\gets$ in Lemma \ref{lem:bijection_simplified}. 
\end{remark}

\begin{remark}
    Applying Eq.~\eqref{eq:simplified_encoding} on underlying graph space is an alternative to achieve weak connectivity, which is simpler comparing to Eq.~\eqref{eq:general_encoding} w.r.t. the number of variables and constraints.
\end{remark}

\section{Symmetry Breaking}\label{sec:symmetry_breaking}
We assumed that all nodes are labeled in Sections \ref{sec:general_encoding} and \ref{sec:simplified_encoding}, which differentiates isomorphic graphs in the search space. In certain scenarios, however, isomorphic graphs correspond to the same solution, causing the so-called symmetry issue. 


Observe that the symmetries result from different indexing of nodes for the same graph. In general, there exist $n!$ ways to index $n$ nodes, each of which corresponds to one solution in our graph encoding. Symmetry significantly enlarges the search space and hinders efficient optimization. We aim to introduce symmetry-breaking techniques to remove symmetries and resolve symmetry issue. Since only graphs with the same node numbers could be isomorphic, we assume that all graphs have $n$ nodes in our analysis, but constraints introduced in this section work for graphs with node number ranging from $n_0$ to $n$ as before.

\subsection{Undirected Graphs}\label{subsec:symmetry_undirected}
We start from connected undirected graphs. For unconnected graphs, one can apply our symmetry-breaking techniques to each connected component. \citet{mcdonald2024mixed} proposed a set of constraints to force molecules to be connected, which we extended to general undirected graphs \citep{zhang2023optimizing}. 
\begin{equation*}
    \begin{aligned}
        \forall v\in [n]\backslash\{0\},~\exists u<v,~s.t.~A_{u,v}=1
    \end{aligned}
\end{equation*}
The linear form of these connectivity constraints is:
\begin{equation}\label{eq:connectivity}
    \begin{aligned}
        \sum\limits_{u<v}A_{u,v}>1,~\forall v\in [n]\backslash\{0\}
    \end{aligned}
\end{equation}
Eq.~\eqref{eq:connectivity} requires each node (except for node $0$) to be linked with at least one node with smaller index, which obviously results in connected graphs. However, the reason for introducing Eq.~\eqref{eq:connectivity} here is not for connectivity but for removing symmetries, since Eq.~\eqref{eq:strong_connectivity} already guarantees connectivity without any restrictions over indexes.

Even though Eq.~\eqref{eq:connectivity} helps break symmetry, it cannot break much symmetry: any neighbor of node $0$ could be indexed $1$, then any neighbor of node $0$ or $1$ could be indexed $2$, and so on. We want to limit the number of possible indexing. A natural idea is to account for neighbors of each node: the neighbor set of a node with smaller index should also have smaller lexicographical order. Before further discussion, we need to define the lexicographical order, which is slightly different from the classic definition. 

\begin{definition}
    Define $Seq(L,M)$ as the set of all non-decreasing sequences with $L$ integer elements in $[0,M]$:
    \begin{equation}
        \begin{aligned}
            Seq(M,L)= \{(a_1,a_2,\dots,a_L)~|~ 0\le a_1 \le a_2 \le\cdots \le a_L \le M\}
        \end{aligned}
    \end{equation}
\end{definition}

\begin{definition}[Lexicographic order]
    For any $a\in Seq(L,M)$, denote its lexicographical order by $LO(a)$ such that for any $a,b\in \mathcal Seq(L,M)$, we have:
    \begin{equation}\label{eq:LO}
        \begin{aligned}
            LO(a)<LO(b)~\Leftrightarrow~\exists \;l \in \{1,2,\dots,L\},~\text{s.t.} 
            \begin{cases}
                a_i=b_i, & 1\le i<l \\
                a_l<b_l &
            \end{cases}
        \end{aligned}
    \end{equation}
\end{definition}

For any multiset $S$ with no more than $L$ integer elements in $[0,M-1]$, we first sort all elements in $S$ in a non-decreasing order, then add elements with value $M$ until the length equals to $L$. In that way, we build an injection mapping $S$ to a unique sequence $s\in Seq(L,M)$. Define the lexicographical order of $S$ by $LO(s)$. For simplicity, we still use $LO(S)$ to denote the lexicographical order of $S$. 

\begin{remark}
    Our definition of $LO(\cdot)$ only differs from the classic definition in cases where $a$ is a prefix of $b$: $LO(a)>LO(b)$, while the lexicographical order of $a$ is smaller than $b$. We define $LO(\cdot)$ in this way for the convenience of its MIP formulation, e.g., Eq.~\eqref{eq:break_symmetry_undirected}.
\end{remark}

Since we only need lexicographical orders for all neighbor sets of the graph, let $L=n-1, M=n$. With the definition of $LO(\cdot)$, we can represent the aforementioned idea by:
\begin{equation}\label{eq:LO_undirected}
    \begin{aligned}
        LO(\mathcal N(v)\backslash\{v+1\})\le LO(\mathcal N(v+1)\backslash \{v\}),~\forall v\in [n-1]
    \end{aligned}
\end{equation}
where $\mathcal N(v)$ is the neighbor set of node $v$. 

Note that the possible edge between nodes $v$ and $v+1$ is ignored in Eq.~\eqref{eq:LO_undirected} to include the cases that they are linked and share all neighbors with indexes less than $v$. In this case, $LO(\mathcal N(v+1))$ is necessarily smaller than $LO(\mathcal N(v))$ since $v\in \mathcal N(v+1)$. 

Eq.~\eqref{eq:LO_undirected} can be equivalently rewritten as:
\begin{equation}\label{eq:break_symmetry_undirected}
    \begin{aligned}
        \sum\limits_{u\neq v,v+1}2^{n-u-1}\cdot A_{u,v} \ge \sum\limits_{u\neq v,v+1}2^{n-u-1}\cdot A_{u,v+1},~\forall v\in [n-1]
    \end{aligned}
\end{equation}
The coefficients $\{2^{n-u-1}\}_{u\in [n]}$ are used to build a bijection between all possible neighbor sets and all integers in $[0,2^n-1]$. These coefficients are also called ``universal ordering vector" in \citep{friedman2007fundamental}.

Eq.~\eqref{eq:LO_undirected} excludes many ways to index nodes, for example it reduces the possible ways to index the graph in Appendix \ref{app:example_undirected} from 720 to 4. But, we still need to ensure that there exists at least one feasible indexing for any graph after applying Eq.~\eqref{eq:LO_undirected}, as shown in the following theorem:

\begin{theorem}\label{thm:feasibility_undirected}
    Given any undirected graph $G=(V,E)$, the indexing yielded from Algorithm \ref{alg:indexing_undirected} satisfies Eq.~\eqref{eq:LO_undirected}.
\end{theorem}

\begin{algorithm}[t]
    \caption{Indexing algorithm for undirected graphs}\label{alg:indexing_undirected}
    \begin{algorithmic}
    \State \textbf{Input:} $G=(V,E)$ with node set $V=\{v_0,v_1,\dots,v_{n-1}\}$. $\mathcal N(v)$ is the neighbor set of node $v\in V$.
    \State $s\gets 0$ \Comment{index starts from 0}
    \State $V_1^1\gets\emptyset$ \Comment{initialize set of indexed nodes}
    \While{$s<n$}
    \State $V_2^s\gets V\backslash V_1^s$ \Comment{set of unindexed nodes}   
    \State $\mathcal N^{s}(v)\gets \{\mathcal I(u)~|~u\in \mathcal N(v)\cap V_1^{s}\},~\forall v\in V_2^{s}$ \Comment{obtain all indexed neighbors}
    \State $rank^s(v)\gets\left|\{LO(\mathcal N^s(u))<LO(\mathcal N^s(v))~|~\forall u\in V_2^s\}\right|,~\forall v\in V_2^s$ 
    \State \Comment{assign a rank to each unindexed node}
    \State $\mathcal I^{s}(v)\gets 
            \begin{cases}
                \mathcal I(v),&\forall v\in V_1^{s}\\
                rank^s(v)+s,&\forall v\in V_2^{s}
            \end{cases}$ \Comment{assign temporary indexes}
    \State $\mathcal N^{s}_{t}(v)\gets \{\mathcal I^{s}(u)~|~u\in\mathcal N(v)\},~\forall v\in V_2^{s}$ \Comment{define temporary neighbor sets based on $\mathcal I^{s}$}
    \State $v^{s}\gets \arg\min\limits_{v\in V_2^{s}} LO(\mathcal N^{s}_{t}(v))$ \Comment{neighbor set of $v^{s}$ has minimal order}
    \State \Comment{if multiple nodes share the same minimal order, arbitrarily choose one}
    \State $\mathcal I(v^s)=s$ \Comment{index $s$ to node $v^s$}  
    \State $V_1^{s+1}\gets V_1^{s}\cup \{v^{s}\}$ \Comment{add $v^{s}$ to set of indexed nodes}   
    \State $s\gets s+1$ \Comment{next index is $s+1$}
    \EndWhile
    \State \textbf{Output:} $\{\mathcal I(v)\}_{v\in V}$ \Comment{result indexing}
    \end{algorithmic}
\end{algorithm}

Theorem \ref{thm:feasibility_undirected} relies on a graph indexing algorithm, i.e., Algorithm \ref{alg:indexing_undirected}, to constructs a feasible indexing. Before proving the theorem, we first give some propositions and then use them to derive an important lemma. 

\begin{proposition}\label{prop:undirected_1}
    For any $s\in [n]$, $\mathcal I^s(v)
    \begin{cases}
        <s,&\forall v\in V_1^s\\
        \ge s,&\forall v\in V_2^s
    \end{cases}$.
\end{proposition}
\begin{proof}
    At the $s$-th iteration, nodes in $V_1^{s}$ have been indexed by $0,1,\dots,s-1$. Therefore, if $v\in V_1^{s}$, then $\mathcal I^{s}(v)=\mathcal I(v)<s$. If $v\in V_2^{s}$, since $\mathcal I^{s}(v)$ is the sum of $s$ and $rank^s(v)$ (which is non-negative), then we have $\mathcal I^{s}(v)\ge s$.
\end{proof}

\begin{proposition}\label{prop:undirected_2}
    For any $s_1,s_2\in [n]$, $s_1\le s_2 \Rightarrow \mathcal N^{s_1}(v)=\mathcal N^{s_2}(v)\cap [s_1],~\forall v\in V_2^{s_2}$
\end{proposition}
\begin{proof}
    $\mathcal N^{s_1}(\cdot)$ is well-defined on $V_2^{s_2}$ since $V_2^{s_2}\subset V_2^{s_1}$. By definitions of $\mathcal N^{s}(\cdot)$ and $V_1^{s}$, for any $v\in V_2^{s_2}$ we have:
    \begin{equation*}
        \begin{aligned}
            \mathcal N^{s_1}(v)=&\{\mathcal I(u)~|~u\in \mathcal N(v)\cap V_1^{s_1}\}&&\longleftarrow\text{$V_1^{s_1}\subset V_1^{s_2}$}\\
            =& \{\mathcal I(u)~|~u\in \mathcal N(v)\cap V_1^{s_1}\cap V_1^{s_2}\}&&\longleftarrow\text{rewrite}\\
            =& \{\mathcal I(u)~|~u\in (\mathcal N(v)\cap V_1^{s_2})\cap V_1^{s_1}\}&&\longleftarrow\text{Proposition \ref{prop:undirected_1}}\\
            =& \{\mathcal I(u)~|~u\in \mathcal N(v)\cap V_1^{s_2}, \mathcal I(u)<s_1\}&&\longleftarrow\text{rewrite}\\
            =& \{\mathcal I(u)~|~u\in \mathcal N(v)\cap V_1^{s_2}\}\cap [s_1]&&\longleftarrow\text{definition of $\mathcal N^{s_2}(v)$}\\
            =&\mathcal N^{s_2}(v)\cap [s_1]
        \end{aligned}
    \end{equation*}
    which completes the proof.
\end{proof}

\begin{proposition}\label{prop:undirected_3}
    Given any multisets $S,T$ with no more than $L$ integer elements in $[0,M-1]$, we have:
    \begin{equation*}
        \begin{aligned}
            LO(S)\le LO(T) \Rightarrow LO(S\cap [m])\le LO(T\cap [m]),~\forall m =1,2,\dots, M
        \end{aligned}
    \end{equation*}
\end{proposition}
\begin{proof}
    By the definition of lexicographical order for multisets, denote the corresponding sequence to $S,T,S\cap [m],T\cap [m]$ by $s,t,s^m,t^m \in Seq(L,M)$, respectively. Then it is equivalent to show that:
    \begin{equation*}
        \begin{aligned}
            LO(s)\le LO(t) \Rightarrow LO(s^m)\le LO(t^m),~\forall m=1,2,\dots,M
        \end{aligned}
    \end{equation*}
    Let $s=(s_1,s_2,\dots,s_L)$ and $t=(t_1,t_2,\dots,t_L)$. If $LO(s)=LO(t)$, then $s=t$ and $s^m=t^m$. Thus $LO(s^m)=LO(t^m)$. Otherwise, if $LO(s)<LO(t)$, then there exists $1\le l\le L$ such that:
    \begin{equation*}
        \begin{aligned}
            \begin{cases}
            s_i=t_i, & \forall 1\le i<l \\
            s_l<t_l &
            \end{cases}
        \end{aligned}
    \end{equation*}
    If $m\le s_l$, then $s^m=t^m$, which means $LO(s^m)=LO(t^m)$. Otherwise, if $m>s_l$, then $s^m$ and $t^m$ share the same first $l-1$ elements. But the $l$-th element of $s^m$ is $s_l$, while the $l$-th element of $t_m$ is either $t_l$ or $M$. In both cases, we have $LO(s^m)<LO(t^m)$.
\end{proof}

Using Propositions \ref{prop:undirected_1}--\ref{prop:undirected_3}, we can prove Lemma \ref{lem:undirected_1}, which shows that if the final index of node $u$ is smaller than $v$, then at each iteration, the temporary index assigned to $u$ is not greater than $v$.

\begin{lemma}\label{lem:undirected_1}
    For any two nodes $u$ and $v$, 
    \begin{equation*}
        \begin{aligned}
            \mathcal I(u)<\mathcal I(v)
            \Rightarrow
            \mathcal I^s(u)\le \mathcal I^s(v),~\forall s\in [n]
        \end{aligned}
    \end{equation*}
\end{lemma}
\begin{proof} Consider three cases as follows:

    \textbf{Case I: } If $\mathcal I(u)<\mathcal I(v)<s$, then:
    \begin{equation*}
        \begin{aligned}
            \mathcal I^s(u)=\mathcal I(u)<\mathcal I(v)=\mathcal I^s(v)
        \end{aligned}
    \end{equation*}
    
    \textbf{Case II: } If $\mathcal I(u)<s\le \mathcal I(v)$, then $\mathcal I^s(u)=\mathcal I(u)$ and:
    \begin{equation*}
        \begin{aligned}
            \mathcal I^s(v)\ge s>\mathcal I(u)=\mathcal I^s(u)
        \end{aligned}
    \end{equation*}
    where Proposition \ref{prop:undirected_1} is used.
    
    \textbf{Case III: } If $s\le \mathcal I(u)<\mathcal I(v)$, at $\mathcal I(u)$-th iteration, $u$ is chosen to be indexed $\mathcal I(u)$. Thus:
    \begin{equation*}
        \begin{aligned}
            LO(\mathcal N^{\mathcal I(u)}_{t}(u))\le LO(\mathcal N^{\mathcal I(u)}_{t}(v))
        \end{aligned}
    \end{equation*}
    According to the definition of $\mathcal N^{\mathcal I(u)}(\cdot)$ and $\mathcal N^{\mathcal I(u)}_t(\cdot)$, we have:
    \begin{equation*}
        \begin{aligned}
            \mathcal N^{\mathcal I(u)}(u)=\mathcal N^{\mathcal I(u)}_t(u)\cap [\mathcal I(u)],
            ~\mathcal N^{\mathcal I(u)}(v)=\mathcal N^{\mathcal I(u)}_t(v)\cap [\mathcal I(u)]
        \end{aligned}
    \end{equation*}
    Using Proposition \ref{prop:undirected_3} (with $S=\mathcal N^{\mathcal I(u)}_t(u)$,~$T=\mathcal N^{\mathcal I(u)}_t(v)$,~$m=\mathcal I(u)$) yields:
    \begin{equation*}
        \begin{aligned}
            LO(\mathcal N^{\mathcal I(u)}(u))\le LO(\mathcal N^{\mathcal I(u)}(v))
        \end{aligned}
    \end{equation*}
    Apply Proposition \ref{prop:undirected_2} (with $s_1=s,s_2=\mathcal I(u)$) for $u$ and $v$, we have:
    \begin{equation*}
        \begin{aligned}
            \mathcal N^s(u)=\mathcal N^{\mathcal I(u)}(u)\cap [s],
            ~\mathcal N^s(v)=\mathcal N^{\mathcal I(u)}(v)\cap [s]
        \end{aligned}
    \end{equation*}
    Using Proposition \ref{prop:undirected_3} again (with $S=N^{\mathcal I(u)}(u),~T=N^{\mathcal I(u)}(v),~m=s$) gives:
    \begin{equation*}
        \begin{aligned}
            LO(\mathcal N^s(u))\le LO(\mathcal N^s(v))
        \end{aligned}
    \end{equation*}
    Recall the definition of $\mathcal I^s(\cdot)$, we have:
    \begin{equation*}
        \begin{aligned}
            \mathcal I^s(u)\le \mathcal I^s(v)
        \end{aligned}
    \end{equation*}
    which completes the proof.
\end{proof}

Now we can prove Theorem \ref{thm:feasibility_undirected}.

\begin{proof}[Proof of Theorem \ref{thm:feasibility_undirected}]
    Assume that Eq.~\eqref{eq:LO_undirected} is not satisfied and the minimal index that violates Eq.~\eqref{eq:LO_undirected} is $s$. Denote nodes with index $s$ and $s+1$ by $u,v$ respectively, i.e., $\mathcal I(u)=s,~\mathcal I(v)=s+1$. 
    
    Let $\mathcal N(u)\backslash \{v\}:=\{u_1,u_2,\dots,u_{n_u}\}$ be all neighbors of $u$ except for $v$, where:
    \begin{equation*}
        \begin{aligned}
            \mathcal I(u_i)<\mathcal I(u_{i+1}),~\forall i=1,2,\dots,n_u-1
        \end{aligned}
    \end{equation*}
    Similarly, let $\mathcal N(v)\backslash \{u\}:=\{v_1,v_2,\dots,v_{n_v}\}$ be all neighbors of $v$ except for $u$, where:
    \begin{equation*}
        \begin{aligned}
            \mathcal I(v_j)<\mathcal I(v_{j+1}),~\forall j=1,2,\dots,n_v-1
        \end{aligned}
    \end{equation*}
    Denote the sequences in $Seq(n-1,n)$ corresponding to sets $\{\mathcal I(u_i)~|~1\le i\le n_u\}$ and $\{\mathcal I(v_j)~|~1\le j\le n_v\}$ by $a=(a_1,a_2,\dots,a_{n-1}),b=(b_1,b_2,\dots,b_{n-1})$. By definition of $LO(\cdot)$:
    \begin{equation*}
        \begin{aligned}
            a_i=
            \begin{cases}
                \mathcal I(u_i),&1\le i\le n_u\\
                n,& n_u<i<n
            \end{cases},~
            b_j=
            \begin{cases}
                \mathcal I(v_j),&1\le j\le n_v\\
                n,& n_v<j<n
            \end{cases}
        \end{aligned}
    \end{equation*}
    
    Since nodes $u$ and $v$ violate Eq.~\eqref{eq:LO_undirected}, there exists a position $1\le k\le n-1$ satisfying:
    \begin{equation*}
        \begin{aligned}
            \begin{cases}
                a_i=b_i, & \forall 1\le i<k \\
                a_k>b_k & 
            \end{cases}
        \end{aligned}
    \end{equation*}
    from where we know that nodes $u$ and $v$ share the first $k-1$ neighbors, i.e. $u_i=v_i,~\forall 1\le i<k$. From $b_k<a_k\le n$ we know that node $v$ definitely has its $k$-th neighbor node $v_k$. Also, note that $v_k$ is not a neighbor of node $u$. Otherwise, we have $u_k=v_k$ and then $a_k=\mathcal I(u_k)=\mathcal I(v_k)=b_k$.
    
    \textbf{Case I: } If $a_k=n$, that is, node $u$ has $k-1$ neighbors. 
    
    In this case, node $v$ has all neighbors of node $u$ as well as node $v_k$. Therefore, we have:
    \begin{equation*}
        \begin{aligned}
            LO(\mathcal N^s_t(u))>LO(\mathcal N^s_t(v))
        \end{aligned}
    \end{equation*}
    which violates the fact that node $u$ is chosen to be indexed $s$ at $s$-th iteration of Algorithm \ref{alg:indexing_undirected}.
    
    \textbf{Case II: } If $a_k<n$, that is, node $u$ has nodes $u_k$ as its $k$-th neighbor. 
    
    Since $\mathcal I(u_k)=a_k>b_k=\mathcal I(v_k)$, apply Lemma \ref{lem:undirected_1} on node $u_k$ and $v_k$ at ($s+1$)-th iteration and obtain that:
    \begin{equation}\label{eq:u_k>=v_k}
        \begin{aligned}
            \mathcal I^{s+1}(u_k)\ge \mathcal I^{s+1}(v_k)
        \end{aligned}
    \end{equation}
    
    Similarly, applying Lemma \ref{lem:undirected_1} to all the neighbors of node $u$ and node $v$ at $s$-th iteration gives:
    \begin{equation*}
        \begin{aligned}
            \mathcal I^s(u_i)&\le \mathcal I^s(u_{i+1}),~\forall i=1,2,\dots,n_u-1 \\
            \mathcal I^s(v_j)&\le \mathcal I^s(v_{j+1}),~\forall j=1,2,\dots,n_v-1 
        \end{aligned}
    \end{equation*}
    Given that $a_k=\mathcal I(u_k)$ is the $k$-th smallest number in $a$, we conclude that $\mathcal I^s(u_k)$ is equal to the $k$-th smallest number in $\mathcal N^s_t(u)$. Likewise, $\mathcal I^s(v_k)$ equals to the $k$-th smallest number in $\mathcal N^s_t(v)$. Meanwhile, $\mathcal I^s(u_i)=\mathcal I^s(v_i)$ since $u_i=v_i,~\forall 1\le i<k$. After comparing the lexicographical orders between of $\mathcal N^s_t(u)$ and $\mathcal N^s_t(v)$ (with the same $k-1$ smallest elements, $\mathcal I^s(u_k)$ and $\mathcal I^s(v_k)$ as the $k$-th smallest element, respectively), node $u$ is chosen. Therefore, we have:
    \begin{equation*}
        \begin{aligned}
            \mathcal I^s(u_k)\le \mathcal I^s(v_k)
        \end{aligned}
    \end{equation*}
    from which we know that:
    \begin{equation*}
        \begin{aligned}
            LO(\mathcal N^s(u_k))\le LO(\mathcal N^s(v_k))
        \end{aligned}
    \end{equation*}
    At ($s+1$)-th iteration, node $u_k$ has one more indexed neighbor, i.e., node $u$ with index $s$, while node $v_k$ has no new indexed neighbor. Thus we have:
    \begin{equation*}
        \begin{aligned}
            LO(\mathcal N^{s+1}(u_k))=LO(\mathcal N^s(u_k)\cup \{s\})< LO(\mathcal N^s(u_k))\le LO(\mathcal N^s(v_k))=LO(\mathcal N^{s+1}(v_k))
        \end{aligned}
    \end{equation*}
    which yields:
    \begin{equation}\label{eq:u_k<v_k}
        \begin{aligned}
            \mathcal I^{s+1}(u_k)< \mathcal I^{s+1}(v_k)
        \end{aligned}
    \end{equation}
    The contradiction between Eq.~\eqref{eq:u_k>=v_k} and Eq.~\eqref{eq:u_k<v_k} completes this proof.
\end{proof}

Theorem \ref{thm:feasibility_undirected} guarantees that there exists at least one indexing satisfying Eq.~\eqref{eq:LO_undirected}. We further show that Eq.~\eqref{eq:LO_undirected} implies Eq.~\eqref{eq:connectivity} on connected graphs, as shown in Lemma \ref{lem:undirected_2}. 

\begin{lemma}\label{lem:undirected_2}
    For any connected undirected graph $G=(V,E)$, if one indexing of $G$ satisfies Eq.~\eqref{eq:LO_undirected}, then it satisfies Eq.~\eqref{eq:connectivity}.
\end{lemma}
\begin{proof}
    Node $0$ itself is a connected graph. Assume that the subgraph induced by nodes $\{0,1,\dots, v\}$ is connected, it suffices to show that the subgraph induced by nodes $\{0,1,\dots, v+1\}$ is connected. Equivalently, we need to prove that there exists $u<v+1$ such that $A_{u,v+1}=1$.

    Assume that $A_{u,v+1}=0,~\forall u<v+1$. Since $G$ is connected, there exists $v'>v+1$ such that:
    \begin{equation*}
        \begin{aligned}
            \exists u<v+1,~\text{s.t.}~A_{u,v'}=1
        \end{aligned}
    \end{equation*}
    Then we know:
    \begin{equation*}
        \begin{aligned}
            \mathcal N(v+1)\cap [v+1]=\emptyset,~\mathcal N(v')\cap [v+1]\neq \emptyset
        \end{aligned}
    \end{equation*}
    Recall the definition of $LO(\cdot)$, we obtain that:
    \begin{equation}\label{eq:v+1>v'}
        \begin{aligned}
            LO(\mathcal N(v+1)\cap [v+1])>LO(\mathcal N(v')\cap [v+1])
        \end{aligned}
    \end{equation}
    
    Since the indexing satisfies Eq.~\eqref{eq:LO_undirected}, then we have:
    \begin{equation*}
        \begin{aligned}
            LO(\mathcal N(w)\backslash\{w+1\})\le LO(\mathcal N(w+1)\backslash \{w\}),~\forall v<w<v'
        \end{aligned}
    \end{equation*}
    Applying Proposition \ref{prop:undirected_3}:
    \begin{equation*}
        \begin{aligned}
            LO((\mathcal N(w)\backslash\{w+1\})\cap [v+1])\le LO((\mathcal N(w+1)\backslash \{w\})\cap [v+1]),~\forall v<w<v'
        \end{aligned}
    \end{equation*}
    Note that $w>v$. Therefore, 
    \begin{equation*}
        \begin{aligned}
            LO(\mathcal N(w)\cap [v+1])\le LO(\mathcal N(w+1)\cap [v+1]),~\forall v<w<v'
        \end{aligned}
    \end{equation*}
    Choosing $w=v+1$ gives:
    \begin{equation}\label{eq:v+1<=v'}
        \begin{aligned}
            LO(\mathcal N(v+1)\cap [v+1])\le LO(\mathcal N(v')\cap [v+1])
        \end{aligned}
    \end{equation}
    The contradiction between Eq.~\eqref{eq:v+1>v'} and Eq.~\eqref{eq:v+1<=v'} completes the proof.
\end{proof}

To break symmetry over digraphs, one can still use Eq.~\eqref{eq:LO_undirected} over underlying graphs, i.e.,
\begin{equation}\label{eq:break_symmetry_directed}
    \begin{aligned}
        \sum\limits_{u\neq v,v+1}2^{n-u-1}\cdot A_{u,v}^U \ge \sum\limits_{u\neq v,v+1}2^{n-u-1}\cdot A_{u,v+1}^U,~\forall v\in [n-1]
    \end{aligned}
\end{equation}

\subsection{Acyclic Digraphs}\label{subsec:symmetry_DAG}
Directions of edges are useful information and could also be utilized to further remove symmetries in some cases. Here we study one of the most popular type of digraphs, i.e., DAGs, and propose more powerful constraints to break symmetry. Similar to Section \ref{subsec:symmetry_undirected}, we assume that all DAGs are weakly connected and have $n$ nodes. The basic idea is similar to Eq.~\eqref{eq:LO_undirected}, but here we consider successor sets that are more informative than neighbor sets.

\begin{definition}[Successor sets]
    Given any acyclic digraph $G=(V,E)$, the successor set of node $v\in V$ is defined as $\mathcal S(v)=\{u\in V\backslash\{v\}~|~\text{node $v$ can reach node $u$}\}$.
\end{definition}

\begin{proposition}\label{prop:successor}
    Given any acyclic digraph $G$. For two different nodes $u$ and $v$, if node $u$ can reach node $v$, then $\mathcal S(v)\subsetneq\mathcal S(u)$.
\end{proposition}
\begin{proof}
    For any node $w$ that can be reached from node $v$, node $u$ can reach node $w$. Thus $\mathcal S(v)\subset\mathcal S(u)$. Meanwhile, since $v\in \mathcal S(u)$ while $v\not\in\mathcal S(v)$, we obtain that $\mathcal S(v)\subsetneq\mathcal S(u)$.
\end{proof}

Motivated by Proposition \ref{prop:successor}, we propose the following symmetry-breaking constraints:
\begin{equation}\label{eq:LO_DAG}
    \begin{aligned}
        LO(\mathcal S(v))\le LO(\mathcal S(v+1)),~\forall v\in [n-1]
    \end{aligned}
\end{equation}
Intuitively, if node $u$ has a successor with smaller index than all successors of node $v$, then node $u$ should have smaller index than node $v$. Note that we do not need to exclude the possible edge between node $v$ and node $v+1$ in Eq.~\eqref{eq:break_symmetry_DAG} like Eq.~\eqref{eq:break_symmetry_undirected}, as this edge must be from node $v$ to node $v+1$ if it exists. 

Using reachability variables $r_{u,v}$, Eq.~\eqref{eq:LO_DAG} can be linearly represented as:
\begin{equation}\label{eq:break_symmetry_DAG}
    \begin{aligned}
        \sum\limits_{u\neq v}2^{n-u-1}\cdot r_{v,u}\ge \sum\limits_{u\neq v+1}2^{n-u-1}\cdot r_{v+1,u},~\forall v\in [n-1]
    \end{aligned}
\end{equation}

Theorem \ref{thm:feasibility_DAG}, relying on Algorithm \ref{alg:indexing_DAG}, guarantees that for any weakly connected DAG, there exists at least one indexing satisfying Eq.~\eqref{eq:LO_DAG}.

\begin{theorem}\label{thm:feasibility_DAG}
    Given any weakly connected DAG $G=(V,E)$, the indexing yielded from Algorithm \ref{alg:indexing_DAG} satisfies Eq.~\eqref{eq:LO_DAG}.
\end{theorem}

\begin{algorithm}[t]
    \caption{Indexing algorithm for DAGs}\label{alg:indexing_DAG}
    \begin{algorithmic}
    \State \textbf{Input:} $G=(V,E)$ with node set $V=\{v_0,v_1,\dots,v_{n-1}\}$. $\mathcal S(v)$ is the successor set of node $v\in V$.
    \State $s\gets n-1$ \Comment{index starts from $n-1$}
    \State $V_1^{n-1}\gets\emptyset$ \Comment{initialize set of indexed nodes}
    \While{$s\ge 0$}
    \State $V_2^s\gets V\backslash V_1^s$ \Comment{set of unindexed nodes}   
    \State $\mathcal I^{s}(v)\gets 
        \begin{cases}
            \mathcal I(v),&\forall v\in V_1^{s}\\
            s,&\forall v\in V_2^{s}
        \end{cases}$ \Comment{assign temporary indexes}
    \State $\mathcal S^{s}_{t}(v)\gets \{\mathcal I^{s}(u)~|~u\in\mathcal S(v)\},~\forall v\in V_2^{s}$ \Comment{define temporary successor sets based on $\mathcal I^{s}$}
    \State $v^{s}\gets \arg\max\limits_{v\in V_2^{s}} LO(\mathcal S^{s}_{t}(v))$ \Comment{successor set of $v^{s}$ has maximal order}
    \State \Comment{if multiple nodes share the same maximal order, arbitrarily choose one}
    \State $\mathcal I(v^s)=s$ \Comment{index $s$ to node $v^s$}  
    \State $V_1^{s-1}\gets V_1^{s}\cup \{v^{s}\}$ \Comment{add $v^{s}$ to set of indexed nodes}   
    \State $s\gets s-1$ \Comment{next index is $s-1$}
    \EndWhile
    \State \textbf{Output:} $\{\mathcal I(v)\}_{v\in V}$ \Comment{result indexing}
    \end{algorithmic}
\end{algorithm}

\begin{proposition}\label{prop:DAG_1}
    For any two different nodes $u$ and $v$, if node $u$ can reach node $v$, then $\mathcal I(u)<\mathcal I(v)$. In other words, $\{\mathcal I(v)\}_{v\in V}$ is a topological ordering.
\end{proposition}
\begin{proof}
    Assume that $I(u)>I(v)$, then we have:
    \begin{equation*}
        \begin{aligned}
            \text{node $u$ can reach node $v$} &\Rightarrow \mathcal S(v)\subsetneq \mathcal S(u)&&\longleftarrow \text{Proposition \ref{prop:successor}}\\
            &\Rightarrow \mathcal S_t^{\mathcal I(u)}(v)\subsetneq \mathcal S_t^{\mathcal I(u)}(u)&&\longleftarrow\text{definition of $\mathcal S^{\mathcal I(u)}(\cdot)$}\\
            &\Rightarrow LO(\mathcal S_t^{\mathcal I(u)}(v))>LO(\mathcal S_t^{\mathcal I(u)}(u))&&\longleftarrow\text{definition of $LO(\cdot)$}
        \end{aligned}
    \end{equation*}
    which contradicts $u=\arg\max_{w\in V_2^{\mathcal I(u)}} LO(\mathcal S_t^{\mathcal I(u)})(w)$.
\end{proof}

\begin{proposition}\label{prop:DAG_2}
    For any node $v$ such that $\mathcal S(v)\neq \emptyset$, let $v_1$ be the successor of $v$ with the smallest index, then we have $\mathcal S_t^{s}(v)=\mathcal S(v),~\forall 0\le s<\mathcal I(v_1)$.
\end{proposition}
\begin{proof}
    By definition of $S_t^s(\cdot)$, we have $\mathcal S(v)\subset V_1^s,~\forall s<\mathcal I(v_1)$. Therefore, $\mathcal S_t^s(v)=\mathcal S(v)$.
\end{proof}
    
Now we prove Theorem \ref{thm:feasibility_DAG} using Propositions \ref{prop:DAG_1}--\ref{prop:DAG_2}.

\begin{proof}[Proof of Theorem \ref{thm:feasibility_DAG}.]
    Assume that Eq.~\eqref{eq:LO_DAG} is not satisfied and the maximal index that violates Eq.~\eqref{eq:LO_DAG} is $s$. Denote nodes with index $s$ and $s+1$ by $u,v$, respectively, i.e., $\mathcal I(u)=s,~\mathcal I(v)=s+1$. By assumption we have $LO(\mathcal S(u))>LO(\mathcal S(v))$, which implies that $\mathcal S(v)\neq \emptyset$.

    \textbf{Case I: } If $\mathcal S(u)=\emptyset$, then $\mathcal S_t^{s+1}(u)=\emptyset$ while $\mathcal S_t^{s+1}(v)\neq \emptyset$. Thus we have:
    \begin{equation*}
        \begin{aligned}
            LO(\mathcal S_t^{s+1}(u))>LO(\mathcal S_t^{s+1}(v))
        \end{aligned}
    \end{equation*}
    which contradicts the assumption that $\mathcal I(v)=s+1$.

    \textbf{Case II:} If $\mathcal S(u)\neq \emptyset$, denote the successor of $u$ with the smallest index as $u_1$. From Proposition \ref{prop:DAG_1} we know that $\mathcal I(u_1)>\mathcal I(u)=s$. If $\mathcal I(u_1)=s+1$, i.e., $u_1=v$, then $LO(\mathcal S(u))<LO(\mathcal S(v))$ since $u$ can reach $v$. Otherwise, if $\mathcal I(u_1)>s+1$, using Proposition \ref{prop:DAG_2} twice gives that:
    \begin{equation*}
        \begin{aligned}
            \mathcal S_t^{s+1}(v)=\mathcal S(v),~\mathcal S_t^{s+1}(u)=\mathcal S(u)
        \end{aligned}
    \end{equation*}
    Then we have that:
    \begin{equation*}
        \begin{aligned}
            LO(\mathcal S_t^{s+1}(v))=LO(\mathcal S(v))<LO(\mathcal S(u))=LO(\mathcal S_t^{s+1}(u))
        \end{aligned}
    \end{equation*}
    which contradicts the assumption that $\mathcal I(v)=s+1$.
\end{proof}

\begin{conjecture}\label{conj:ancestor}
    Eq.~\eqref{eq:LO_DAG} can not order two nodes that share the same successor sets. We conjecture that there exists one indexing satisfying both Eq.~\eqref{eq:LO_DAG} and the following constraints:
    \begin{equation}\label{eq:LO_DAG_plus}
        \begin{aligned}
            LO(\mathcal S(v))=LO(\mathcal S(v+1))\Rightarrow LO(\mathcal A(v))\le LO(\mathcal A(v+1)),~\forall v\in [n-1]
        \end{aligned}
    \end{equation}
    where $\mathcal A(v)$ is the ancestor set of node $v$, i.e., $\mathcal A(v)=\{u\in V\backslash\{v\}~|~\text{node $u$ can reach node $v$}\}$. Eq.~\eqref{eq:LO_DAG_plus} could be rewritten as the following linear form with $2^n$ as a big-M coefficient:
    \begin{equation}\label{eq:break_symmetry_DAG_plus}
        \begin{aligned}
            \sum\limits_{u\neq v}2^{n-u-1}\cdot r_{u,v}\ge\sum\limits_{u\neq v+1}2^{n-u-1}\cdot r_{u,v+1}-2^n\cdot \left(\sum\limits_{u\neq v}2^{n-u-1}\cdot r_{v,u}-\sum\limits_{u\neq v+1}2^{n-u-1}\cdot r_{v+1,u}\right)
        \end{aligned}
    \end{equation}
    However, we have not figured out a general way to construct a feasible indexing. Numerically, we verify the correctness of Eq.~\eqref{eq:LO_DAG_plus} for $n\le 7$, i.e., for any weakly connected DAG with no more than $7$ nodes, there exists an indexing satisfying both Eq.~\eqref{eq:LO_DAG} and Eq.~\eqref{eq:LO_DAG_plus}.
\end{conjecture}

\subsection{Discussion}\label{subsec:discussion}

Symmetry handling is vital in MIP, and different approaches are proposed via fundamental domains \citep{margot2009symmetry,verschae2023geometry}, symmetry-handling constraints \citep{friedman2007fundamental,kaibel2008packing,liberti2008automatic,kaibel2011orbitopal,liberti2012reformulations,liberti2014stabilizer,hojny2019polytopes,hojny2020packing,hojny2023impact}, and variable domain reductions \citep{margot2002pruning,margot2003exploiting,ostrowski2011orbital}, etc. Our approaches fit in the constraints-based symmetry-handling literature, as we focus on the theoretical correctness of the proposed lexicographic constraints, i.e., Eq.~\eqref{eq:LO_undirected} and Eq.~\eqref{eq:LO_DAG}. To show their efficiency, we represent them into linear forms, i.e., Eq.~\eqref{eq:break_symmetry_undirected} and Eq.~\eqref{eq:break_symmetry_DAG}, and conduct experiments in small scales in Section \ref{sec:results}. However, as the graph size increases, those power of two coefficients might cause numerical instabilities. An alternative is propagating lexicographic constraints during branch-and-bound \citep{van2024unified,hojny2025detecting}.

The closest setting to ours in symmetry-handling literature is distributing $m$ different jobs to $n$ identical machines and then minimizing the total cost. Binary variable $A_{i,j}$ denotes if job $i$ is assigned to machine $j$. The requirement is that each job can only be assigned to one machine (but each machine can be assigned to multiple jobs). Symmetries come from all permutations of machines. This setting appears in noise dosage problems \citep{sherali2001improving,ghoniem2011defeating}, packing and partitioning orbitopes \citep{kaibel2008packing,faenza2009extended}, and scheduling problems \citep{ostrowski2010symmetry}. However, requiring that the sum of each row in $A_{i,j}$ equals to $1$ simplifies the problem. By forcing decreasing lexicographical orders for all columns, the symmetry issue is handled well. Eq.~\eqref{eq:LO_undirected} can be regarded as a non-trivial generalization of these constraints from a bipartite graph to an arbitrary undirected graph: following Algorithm \ref{alg:indexing_undirected} will produce the same indexing documented in \cite{sherali2001improving,ghoniem2011defeating,kaibel2008packing,faenza2009extended,ostrowski2010symmetry}.

\begin{table}[]
    \centering
    \caption{Performance of symmetry-breaking constraints. For DAGs, we test both Eq.~\eqref{eq:break_symmetry_directed} and Eq.~\eqref{eq:break_symmetry_DAG}.}
    \label{tab:symmetry_breaking_performance_general}
    \begin{tabular}{cccrrrrrr}
        \toprule
         Space & Setting & OEIS & $n=3$ & $n=4$ & $n=5$ & $n=6$ & $n=7$ \\
        \midrule
        \multirow{3}{*}{$\mathcal G_{n,n}^{C}$} & labeled & \href{https://oeis.org/A001187}{\texttt{A001187}} & 4 & 38 & 728 & 26,704 & 1,866,256\\
                              & Eq.~\eqref{eq:break_symmetry_undirected} & N/A & 2 & 6 & 31 & 262 & 3,628\\
                              & unlabeled & \href{https://oeis.org/A001349}{\texttt{A001349}} & 2 & 6 & 21 & 112 & 853\\
        \midrule
        \multirow{3}{*}{$\mathcal G_{n,n}^{S}$} & labeled & \href{https://oeis.org/A003030}{\texttt{A003030}} & 18 & 1,606 & 565,080 & 734,774,776 & - \\
                              & Eq.~\eqref{eq:break_symmetry_directed} & N/A & 16 & 720 & 84,481 & -  & - \\
                              & unlabeled & \href{https://oeis.org/A035512}{\texttt{A035512}} & 5 & 83 & 5,048 & 1,047,008 & -\\
        \midrule
        \multirow{3}{*}{$\mathcal G_{n,n}^{W}$} & labeled & \href{https://oeis.org/A003027}{\texttt{A003027}} & 54 & 3,834 & 1,027,080 & 1,067,308,488 & - \\
                              & Eq.~\eqref{eq:break_symmetry_directed} & N/A & 36 & 1,188 & 113,157 & - & - \\
                              & unlabeled & \href{https://oeis.org/A003085}{\texttt{A003085}} & 13 & 199 & 9,364 & 1,530,843  & -\\
        \midrule
        \multirow{5}{*}{$\mathcal G_{n,n}^{D,W}$} & labeled & \href{https://oeis.org/A082402}{\texttt{A082402}} & 18 & 446 & 26,430 & 3,596,762 & 1,111,506,858 \\
                              & Eq.~\eqref{eq:break_symmetry_directed} & N/A & 10 & 84 & 1,312 & 39,846 & - \\
                              & Eq.~\eqref{eq:break_symmetry_DAG} & N/A & 4 & 31 & 450 & 12,175 & 627,846 \\
                              & Eqs.~\eqref{eq:break_symmetry_DAG},~\eqref{eq:break_symmetry_DAG_plus} & N/A & 4 & 26 & 326 & 7,769 & 359,396 \\
                              & unlabeled & \href{https://oeis.org/A101228}{\texttt{A101228}} & 4 & 24 & 267 & 5,647 & 237,317\\
        \midrule
        \multirow{5}{*}{$\mathcal G_{n,n}^{D,st}$} & labeled & \href{https://oeis.org/A165950}{\texttt{A165950}} & 12 & 216 & 10,600 & 1,306,620 & 384,471,444 \\
                              & Eq.~\eqref{eq:break_symmetry_directed} & N/A & 8 & 56 & 696 & 17,620 & 978,548 \\
                              & Eq.~\eqref{eq:break_symmetry_DAG} & N/A & 2 & 10 & 114 & 2,730 & 132,978  \\
                              & Eqs.~\eqref{eq:break_symmetry_DAG},~\eqref{eq:break_symmetry_DAG_plus} & N/A & 2 & 10 & 106 & 2,314 & 102,538 \\
                              & unlabeled & \href{https://oeis.org/A345258}{\texttt{A345258}} & 2 & 10 & 98 & 1,960 & 80,176\\
        \bottomrule
    \end{tabular}
    
\end{table}

\section{Numerical Results}\label{sec:results}
In this section, we evaluate the performance of symmetry-breaking constraints introduced in Section \ref{sec:symmetry_breaking} over various weakly connected graph spaces that are unlabeled. For each space, we count the number of feasible solutions under three settings: (a) labeled graphs, i.e., without breaking symmetry, (b) break symmetry using our constraints, and (c) unlabeled graphs, i.e., breaking all symmetry. We present results for both (a) and (c) from their corresponding OEIS. Setting (a) could also be obtained from our graph encoding as shown in Table \ref{tab:summary_encoding}. For setting (b), we use Gurobi \citep{gurobi2024} to count all feasible solutions by setting Gurobi parameter \textrm{PoolSearchMode=$2$}. If the number of feasible solutions exceeds one million, we will not count them due to its high consumption of time and memory (by setting Gurobi parameter \textrm{PoolSolutions=$10^6$}).

\section{Conclusion}\label{sec:conclusion}
This work studies how to represent graph space with corresponding reachability and shortest paths information. A general encoding and its simplified version with proper assumption are proposed. We explicitly show that both encoding could be easily restricted to specific graph types by incorporating extra constraints. Symmetry issues arising from node indexing are resolved by lexicographic constraints. The correctness and completeness of our proposed symmetry-breaking constraints are theoretically guaranteed, and experiments on small scales (where we can exhaust all feasible graphs) show their efficiency. For future work, we might consider (i) graph formulations with more graph properties encoded, (ii) alternatives of implementing the lexicographic constraints to avoid the numerical instabilities as the graph size increases, and (iii) more symmetry-breaking techniques in graph search space. 

\section*{Acknowledgments}
This work was supported by the Engineering and Physical Sciences Research Council [grant no. EP/W003317/1], BASF SE, Ludwigshafen am Rhein to SZ, and a BASF/RAEng Research Chair in Data-Driven Optimisation to RM. RM holds concurrent appointments as a Professor at Imperial and as an Amazon Scholar. This paper describes work performed at Imperial prior to joining Amazon and is not associated with Amazon.

\bibliographystyle{abbrvnat}
\bibliography{main}

\newpage
\appendix
\section{Example of Indexing Undirected Graphs}\label{app:example_undirected}
Given a graph with $N=6$ nodes as shown in Figure \ref{fig:undirected_step_0}, we use Algorithm \ref{alg:indexing_undirected} to index all nodes step by step for illustration. 

Before indexing, we first calculate the neighbor sets for each node:
\begin{equation*}
    \begin{aligned}
        &\mathcal N(v_0)=\{v_1,v_2,v_3,v_4,v_5\},~\mathcal N(v_1)=\{v_0,v_2,v_3,v_4\},~\mathcal N(v_2)=\{v_0,v_1,v_5\}\\
        &\mathcal N(v_3)=\{v_0,v_1,v_4\},~\mathcal N(v_4)=\{v_0,v_1,v_3\},~\mathcal N(v_5)=\{v_0,v_2\}
    \end{aligned}
\end{equation*}

$\bm {s=0:}$ $V_1^0=\emptyset,~V_2^0=\{v_0,v_1,v_2,v_3,v_4,v_5\}$

\begin{figure}[h]
    \centering
    \includegraphics[scale=0.8]{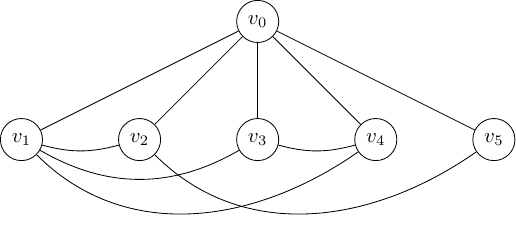}
    \caption{$V_1^0=\emptyset$.}
    \label{fig:undirected_step_0}
\end{figure}

\noindent Obtain indexed neighbors to each unindexed node:
\begin{equation*}
    \begin{aligned}
        \mathcal N^0(v_0)=\mathcal N^0(v_1)=\mathcal N^0(v_2)=\mathcal N^0(v_3)=\mathcal N^0(v_4)=\mathcal N^0(v_5)=\emptyset
    \end{aligned}
\end{equation*}
Rank all unindexed nodes:
\begin{equation*}
    \begin{aligned}
        rank(v_0)=rank(v_1)=rank(v_2)=rank(v_3)=rank(v_4)=rank(v_5)=0
    \end{aligned}
\end{equation*}
Assign a temporary index to each node based on previous indexes (for indexed nodes) and ranks (for unindexed nodes):
\begin{equation*}
    \begin{aligned}
        \mathcal I^0(v_0)=\mathcal I^0(v_1)=\mathcal I^0(v_2)=\mathcal I^0(v_3)=\mathcal I^0(v_4)=\mathcal I^0(v_5)=0
    \end{aligned}
\end{equation*}
After having indexes for all nodes, define temporary neighbor sets:
\begin{equation*}
    \begin{aligned}
        &\mathcal N^0_t(v_0)=\{0,0,0,0,0\},~\mathcal N^0_t(v_1)=\{0,0,0,0\},~\mathcal N^0_t(v_2)=\{0,0,0\},\\
        &\mathcal N^0_t(v_3)=\{0,0,0\},~\mathcal N^0_t(v_4)=\{0,0,0\},~\mathcal N^0_t(v_5)=\{0,0\}
    \end{aligned}
\end{equation*}
Based on the temporary neighbor sets, $v_0$ is chosen to be indexed $0$, i.e., $\mathcal I(v_0)=0$.

$\bm {s=1:}$ $V_1^1=\{v_0\},~V_2^1=\{v_1,v_2,v_3,v_4,v_5\}$

\begin{figure}[h]
    \centering
    \includegraphics[scale=0.8]{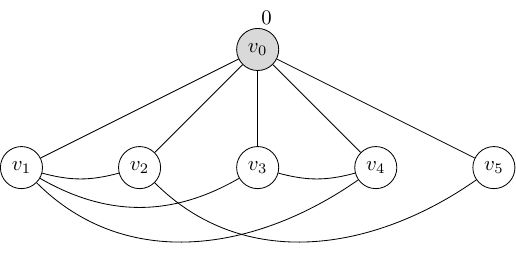}
    \caption{$V_1^1=\{v_0\}$.}
    \label{fig:undirected_step_1}
\end{figure}

\noindent Obtain indexed neighbors to each unindexed node:
\begin{equation*}
    \begin{aligned}
        \mathcal N^1(v_1)=\mathcal N^1(v_2)=\mathcal N^1(v_3)=\mathcal N^1(v_4)=\mathcal N^1(v_5)=\{0\}
    \end{aligned}
\end{equation*}
Rank all unindexed nodes:
\begin{equation*}
    \begin{aligned}
        rank(v_1)=rank(v_2)=rank(v_3)=rank(v_4)=rank(v_5)=0
    \end{aligned}
\end{equation*}
Assign a temporary index to each node based on previous indexes (for indexed nodes) and ranks (for unindexed nodes):
\begin{equation*}
    \begin{aligned}
        \mathcal I^1(v_1)=\mathcal I^1(v_2)=\mathcal I^1(v_3)=\mathcal I^1(v_4)=\mathcal I^1(v_5)=1
    \end{aligned}
\end{equation*}
After having indexes for all nodes, define temporary neighbor sets:
\begin{equation*}
    \begin{aligned}
        &\mathcal N^1_t(v_1)=\{0,1,1,1\},~\mathcal N^1_t(v_2)=\{0,1,1\},~\mathcal N^1_t(v_3)=\{0,1,1\},\\
        &\mathcal N^1_t(v_4)=\{0,1,1\},~\mathcal N^1_t(v_5)=\{0,1\}
    \end{aligned}
\end{equation*}
Based on the temporary neighbor sets, $v_1$ is chosen to be indexed $1$, i.e., $\mathcal I(v_1)=1$.

$\bm {s=2:}$ $V_1^2=\{v_0,v_1\},~V_2^2=\{v_2,v_3,v_4,v_5\}$

\begin{figure}[h]
    \centering
    \includegraphics[scale=0.8]{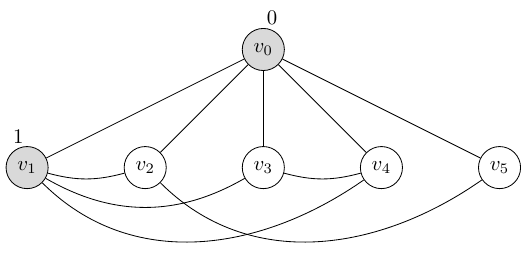}
    \caption{$V_1^2=\{v_0,v_1\}$.}
    \label{fig:undirected_step_2}
\end{figure}

\noindent Obtain indexed neighbors to each unindexed node:
\begin{equation*}
    \begin{aligned}
        \mathcal N^2(v_2)=\mathcal N^2(v_3)=\mathcal N^2(v_4)=\{0,1\},~\mathcal N^2(v_5)=\{0\}
    \end{aligned}
\end{equation*}
Rank all unindexed nodes:
\begin{equation*}
    \begin{aligned}
        rank(v_2)=rank(v_3)=rank(v_4)=0,~rank(v_5)=1
    \end{aligned}
\end{equation*}
Assign a temporary index to each node based on previous indexes (for indexed nodes) and ranks (for unindexed nodes):
\begin{equation*}
    \begin{aligned}
        \mathcal I^2(v_2)=\mathcal I^2(v_3)=\mathcal I^2(v_4)=2,~\mathcal I^2(v_5)=3
    \end{aligned}
\end{equation*}
After having indexes for all nodes, define temporary neighbor sets:
\begin{equation*}
    \begin{aligned}
        \mathcal N^2_t(v_2)=\{0,1,3\},~\mathcal N^2_t(v_3)=\mathcal N^2_t(v_4)=\{0,1,2\},~\mathcal N^2_t(v_5)=\{0,2\}
    \end{aligned}
\end{equation*}
Based on the temporary neighbor sets, both $v_3$ and $v_4$ can be chosen to be indexed $2$, yielding 2 feasible indexing ways. Without loss of generality, let $\mathcal I(v_3)=2$. 

\begin{remark}
    This step explains why temporary indexes and neighbor sets should be added into Algorithm \ref{alg:indexing_undirected}. Otherwise, $v_2$ is also valid to be index $2$, following which $v_3,v_4,v_5$ will be indexed $3,4,5$. Then the neighbor set for $\mathcal I(v_2)=2$ is $\{0,1,5\}$ while the neighbor set for $\mathcal I(v_3)=3$ is $\{0,1,4\}$, which violates constraints Eq.~ \eqref{eq:LO_undirected}.
\end{remark}

$\bm {s=3:}$ $V_1^3=\{v_0,v_1,v_3\},~V_2^3=\{v_2,v_4,v_5\}$

\begin{figure}[h]
    \centering
    \includegraphics[scale=0.8]{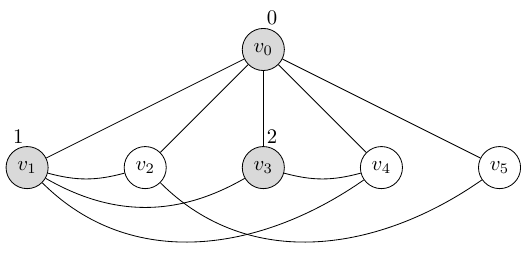}
    \caption{$V_1^3=\{v_0,v_1,v_3\}$.}
    \label{fig:undirected_step_3}
\end{figure}

\noindent Obtain indexed neighbors to each unindexed node:
\begin{equation*}
    \begin{aligned}
        \mathcal N^3(v_2)=\{0,1\},~\mathcal N^3(v_4)=\{0,1,2\},~\mathcal N^3(v_5)=\{0\}
    \end{aligned}
\end{equation*}
Rank all unindexed nodes:
\begin{equation*}
    \begin{aligned}
        rank(v_2)=1,~rank(v_4)=0,~rank(v_5)=2
    \end{aligned}
\end{equation*}
Assign a temporary index to each node based on previous indexes (for indexed nodes) and ranks (for unindexed nodes):
\begin{equation*}
    \begin{aligned}
        \mathcal I^3(v_2)=4,~\mathcal I^3(v_4)=3,~\mathcal I^3(v_5)=5
    \end{aligned}
\end{equation*}
After having indexes for all nodes, define temporary neighbor sets:
\begin{equation*}
    \begin{aligned}
        \mathcal N^3_t(v_2)=\{0,1,5\},~\mathcal N^3_t(v_4)=\{0,1,2\},~\mathcal N^3_t(v_5)=\{0,4\}
    \end{aligned}
\end{equation*}
Based on the temporary neighbor sets, $v_4$ is chosen to be indexed $3$, i.e., $\mathcal I(v_4)=3$.

$\bm {s=4:}$ $V_1^4=\{v_0,v_1,v_3,v_4\},~V_2^4=\{v_2,v_5\}$

\begin{figure}[h]
    \centering
    \includegraphics[scale=0.8]{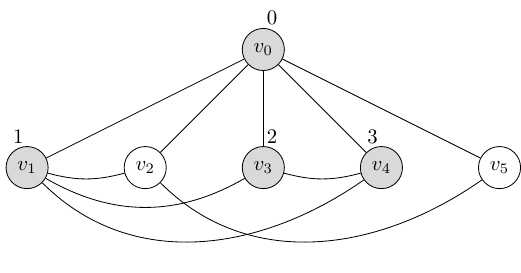}
    \caption{$V_1^4=\{v_0,v_1,v_3,v_4\}$.}
    \label{fig:undirected_step_4}
\end{figure}

\noindent Obtain indexed neighbors to each unindexed node:
\begin{equation*}
    \begin{aligned}
        \mathcal N^4(v_2)=\{0,1\},~\mathcal N^4(v_5)=\{0\}
    \end{aligned}
\end{equation*}
Rank all unindexed nodes:
\begin{equation*}
    \begin{aligned}
        rank(v_2)=0,~rank(v_5)=1
    \end{aligned}
\end{equation*}
Assign a temporary index to each node based on previous indexes (for indexed nodes) and ranks (for unindexed nodes):
\begin{equation*}
    \begin{aligned}
        \mathcal I^4(v_2)=4,~\mathcal I^4(v_5)=5
    \end{aligned}
\end{equation*}
After having indexes for all nodes, define temporary neighbor sets:
\begin{equation*}
    \begin{aligned}
        \mathcal N^4_t(v_2)=\{0,1,5\},~\mathcal N^4_t(v_5)=\{0,4\}
    \end{aligned}
\end{equation*}
Based on the temporary neighbor sets, $v_2$ is chosen to be indexed $4$, i.e., $\mathcal I(v_2)=4$.

$\bm {s=5:}$ $V_1^5=\{v_0,v_1,v_2,v_3,v_4\},~V_2^5=\{v_5\}$

\begin{figure}[h]
    \centering
    \includegraphics[scale=0.8]{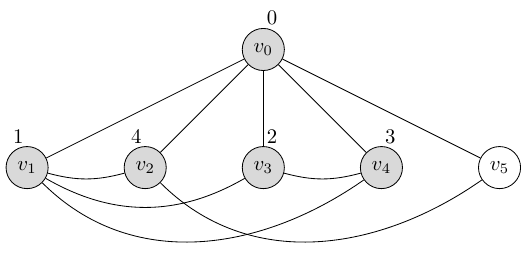}
    \caption{$V_1^5=\{v_0,v_1,v_2,v_3,v_4\}$.}
    \label{fig:undirected_step_5}
\end{figure}

\noindent Since there is only node $v_5$ unindexed, without running the algorithm we still know that $v_5$ is chosen to be indexed $5$, i.e., $\mathcal I(v_5)=5$.

\textbf{Output:}

\begin{figure}[h]
    \centering
    \includegraphics[scale=0.8]{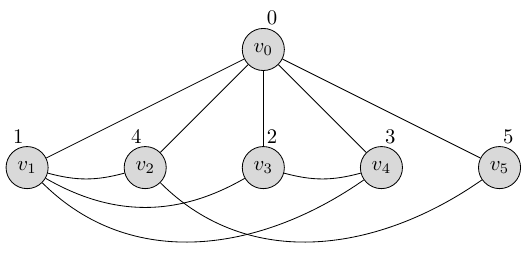}
    \caption{$V_1^6=\{v_0,v_1,v_2,v_3,v_4,v_5\}$.}
    \label{fig:undirected_step_6}
\end{figure}

\section{Example of Indexing DAGs}\label{app:example_DAG}

Given a DAG with $N=7$ nodes as shown in Figure \ref{fig:DAG_step_6}, we use Algorithm \ref{alg:indexing_DAG} to index all nodes. For simplicity, we only present several key steps.

$\bm {s=6:}$ Due to the definition of $\mathcal I^s(\cdot)$, one can easily observe that only nodes without unindexed successors could be chosen as the next node to be indexed. In the first step, only the sink $v_6$ does not have any unindexed successors, which will be indexed $6$, i.e., $\mathcal I(v_6)=6$.

\begin{figure}[h]
    \centering
    \includegraphics[scale=0.8]{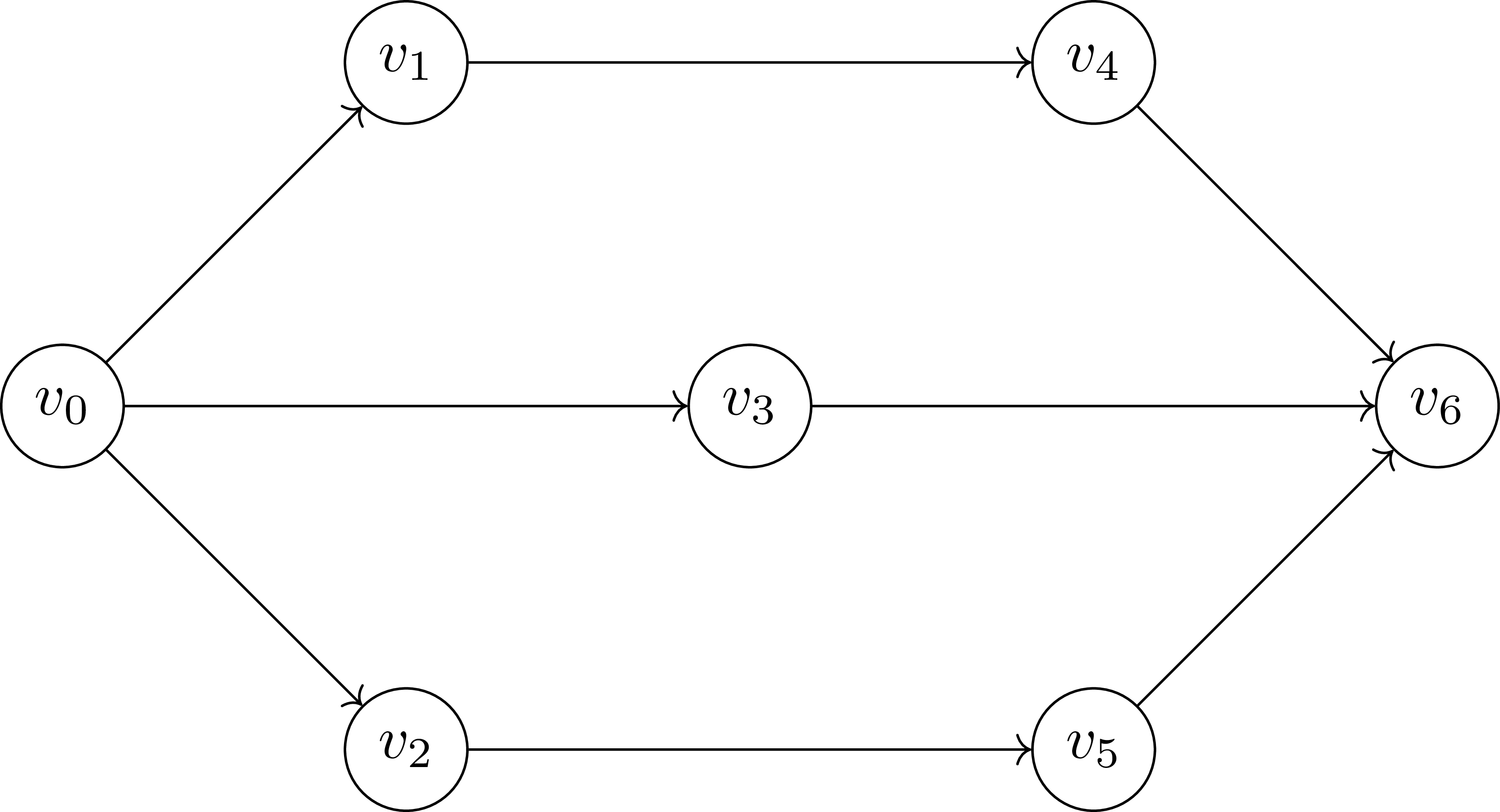}
    \caption{$V_1^6=\emptyset$.}
    \label{fig:DAG_step_6}
\end{figure}

$\bm{s=5,4,3:}$ We combine these 3 steps since $v_3,v_4,v_5$ share the same successor set and all other nodes have at least one successor with index smaller than $6$.

\begin{figure}[h]
    \centering
    \includegraphics[scale=0.8]{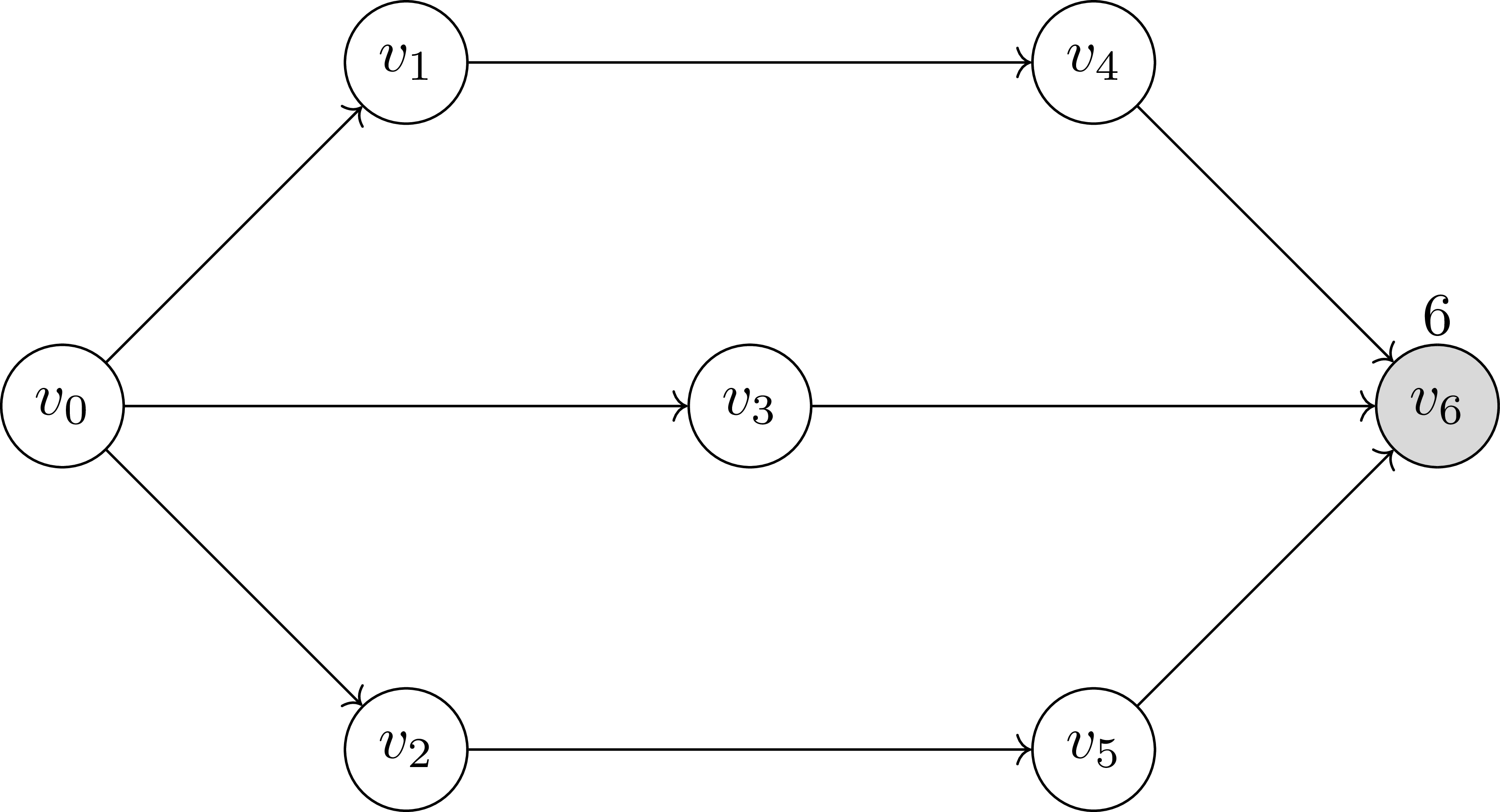}
    \caption{$V_1^5=\{v_6\}$.}
    \label{fig:DAG_step_5}
\end{figure}

\noindent Define temporary successor sets as:
\begin{equation*}
    \begin{aligned}
        &\mathcal S_t^s(v_0)=\{\mathcal I^s(v_1),\mathcal I^s(v_2),\mathcal I^s(v_3),\mathcal I^s(v_4),\mathcal I^s(v_5),6\},~\mathcal S_t^s(v_1)=\{\mathcal I^s(v_4),6\},\\
        &\mathcal S_t^s(v_2)=\{\mathcal I^s(v_5),6\},
        ~\mathcal S_t^s(v_3)=\mathcal S_t^s(v_4)=\mathcal S_t^s(v_5)=\{6\}
    \end{aligned}
\end{equation*}
Since $\mathcal I^s(v_1)=\mathcal I^s(v_2)=s<6,~\mathcal I^s(v_3)<6,~\mathcal I^s(v_4)<6,~\mathcal I^s(v_5)<6$ for $s=5,4,3$, $v_3,v_4,v_5$ will be indexed $s=5,4,3$ in any order, which is the major motivation of Conjecture \ref{conj:ancestor}. Here we arbitrarily choose $\mathcal I(v_4)=5,~\mathcal I(v_3)=4,~\mathcal I(v_5)=3$.

$\bm{s=2,1:}$ We combine these 2 steps since $v_0$ could not be indexed until both $v_1$ and $v_2$ are indexed. 

\begin{figure}[h]
    \centering
    \includegraphics[scale=0.8]{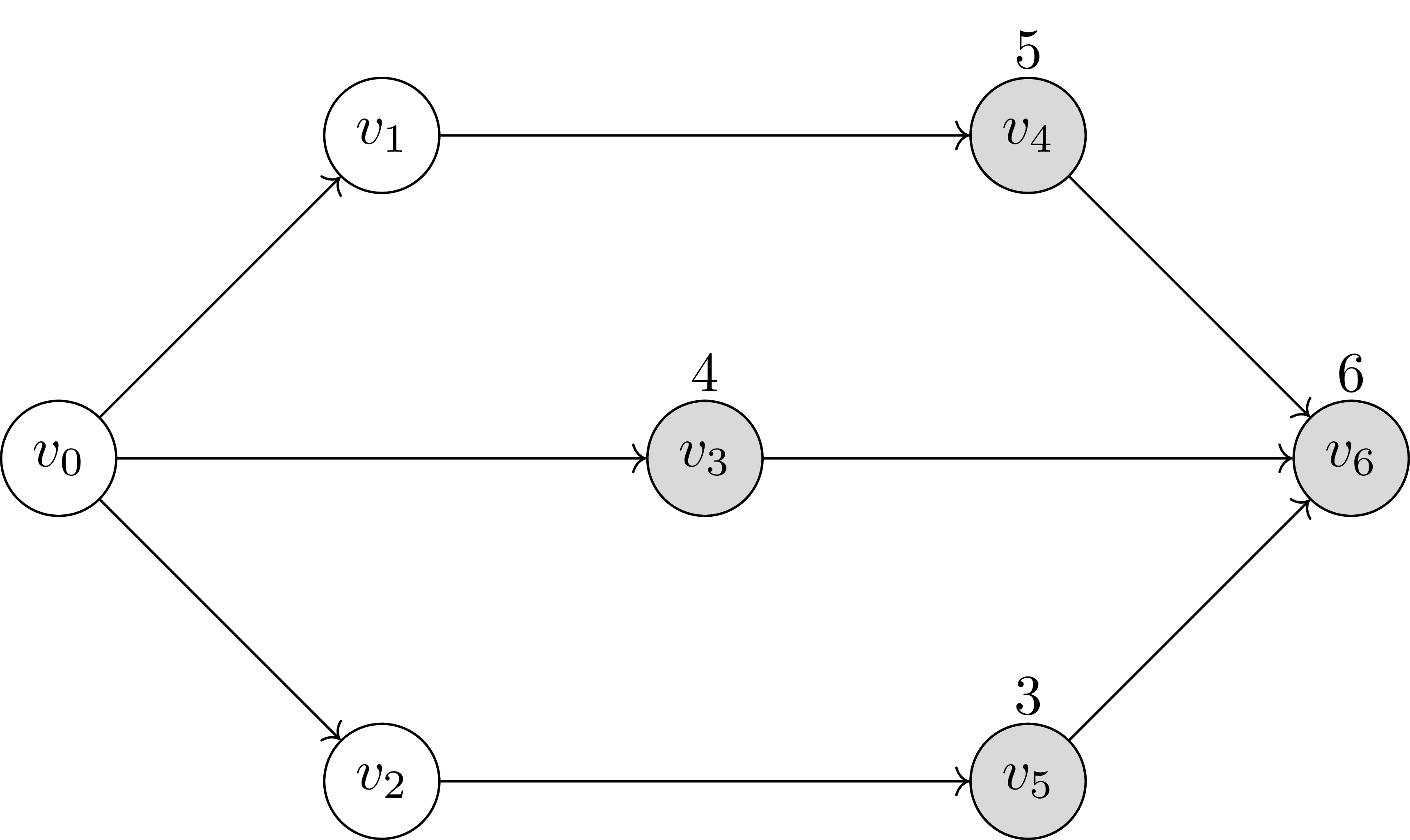}
    \caption{$V_1^2=\{v_3,v_4,v_5,v_6\}$.}
    \label{fig:DAG_step_2}
\end{figure}

\noindent Their temporary successor sets are:
\begin{equation*}
    \begin{aligned}
        \mathcal S_t^s(v_1)=\{5,6\},~\mathcal S_t^s(v_2)=\{3,6\}
    \end{aligned}
\end{equation*}

Since $LO(\mathcal S_t^s(v_1))>LO(\mathcal S_t^s(v_2))$, we have $\mathcal I(v_1)=2,~\mathcal I(v_2)=1$.

$\bm{s=0:}$ Since there is only $v_0$ unindexed, it will be indexed $0$, i.e., $\mathcal I(v_0)=0$.

\begin{figure}[h]
    \centering
    \includegraphics[scale=0.8]{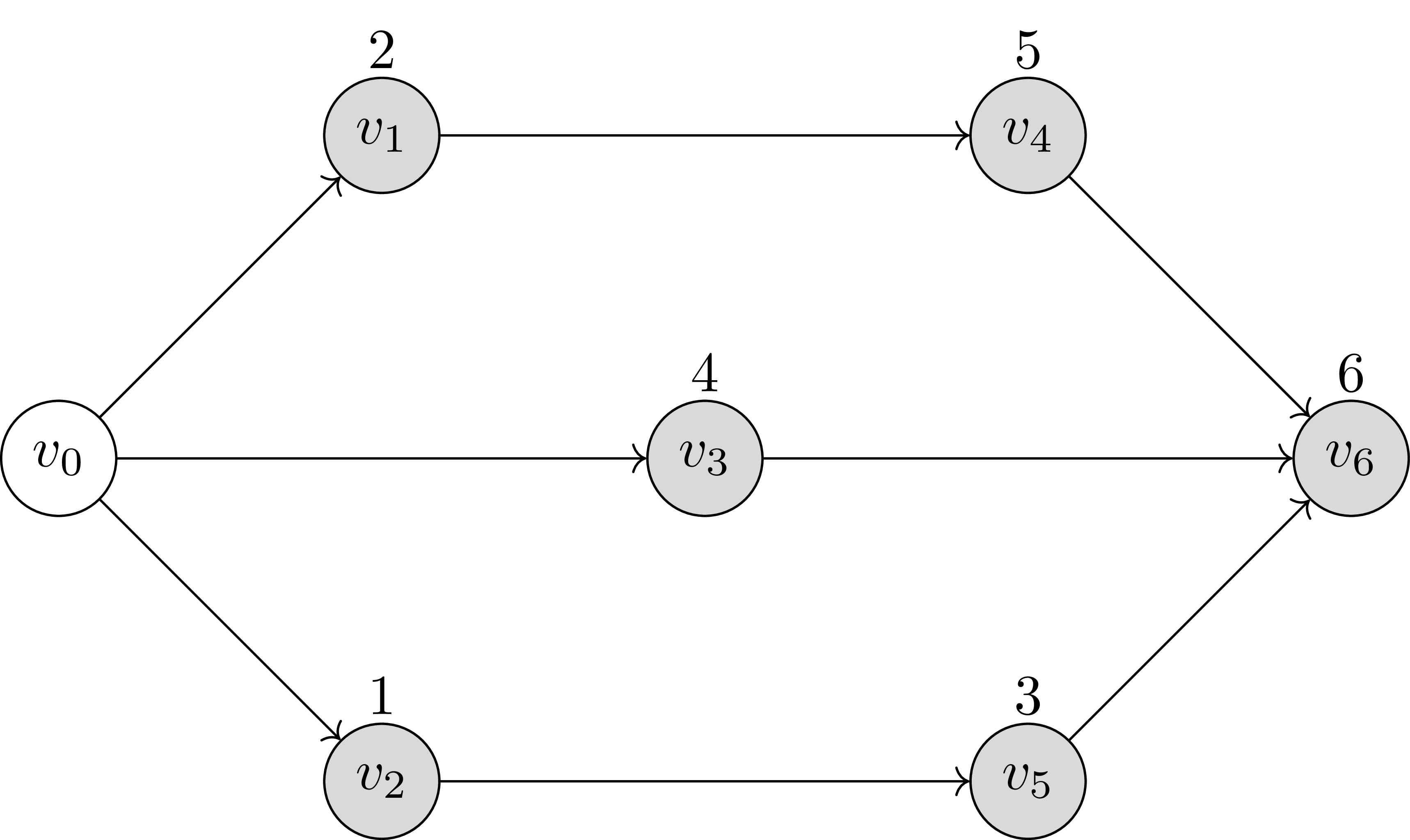}
    \caption{$V_1^{0}=\{v_1,v_2,v_3,v_4,v_5,v_6\}$.}
    \label{fig:DAG_step_0}
\end{figure}

\textbf{Output:}

\begin{figure}[h]
    \centering
    \includegraphics[scale=0.8]{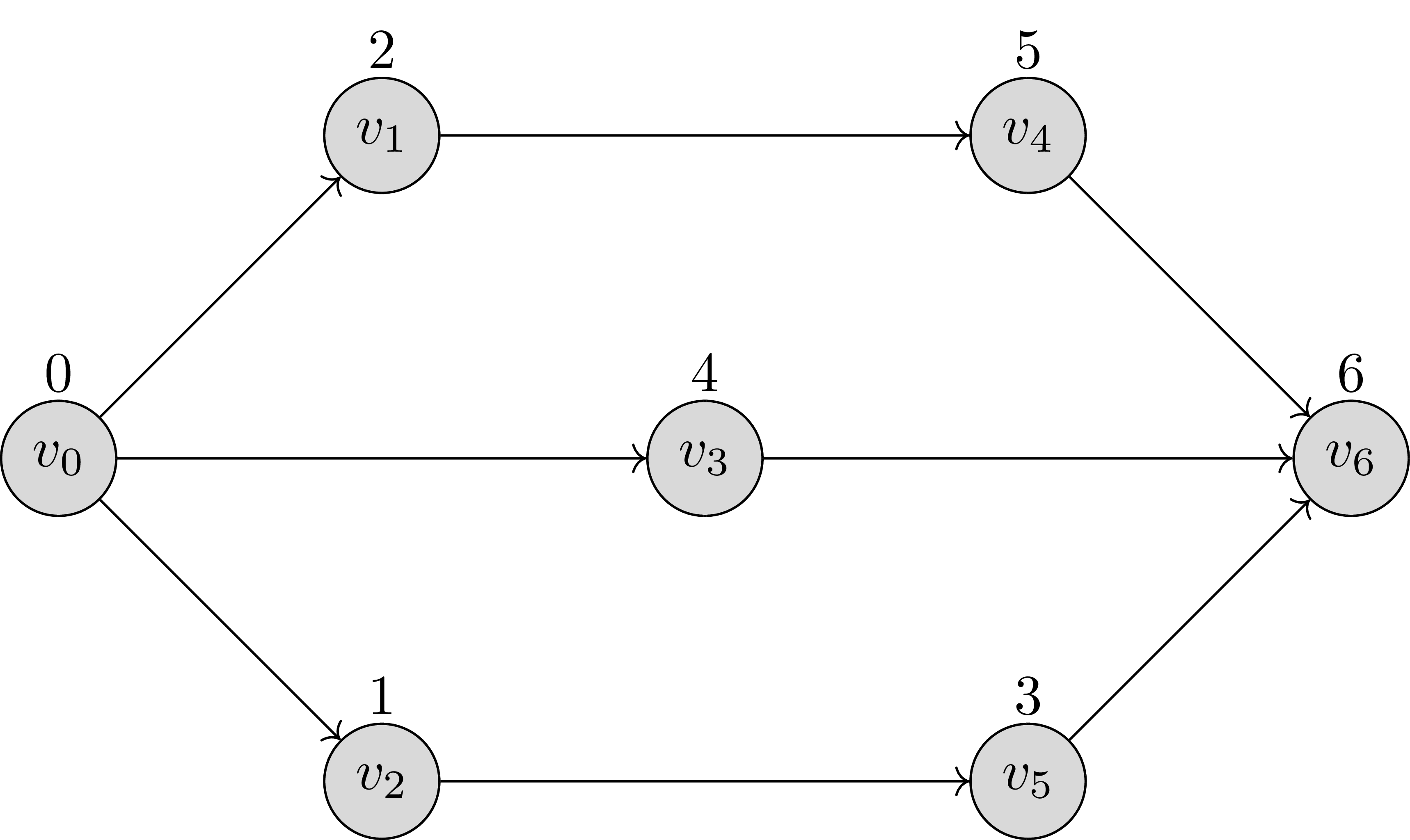}
    \caption{$V_1^{-1}=\{v_0,v_1,v_2,v_3,v_4,v_5,v_6\}$.}
    \label{fig:DAG_step_-1}
\end{figure}

\end{document}